\documentclass[11pt]{article}
\usepackage{amsmath, amssymb, amscd, epsfig, amsthm}
\usepackage{graphicx}
\usepackage{caption, subcaption}

\newtheorem{thm}{Theorem}
\newtheorem{prop}[thm]{Proposition}
\newtheorem{lem}[thm]{Lemma}

\newtheorem{cor}[thm]{Corollary}
\newtheorem{rem}[thm]{Remark}

\renewcommand{\epsilon}{\varepsilon}
\renewcommand{\phi}{\varphi}
\renewcommand{\deg}{\operatorname{deg}}

\newcommand{\BB}{\mathbb}

\newcommand{\separate}{\vskip5pt}

\newcommand{\re}{\operatorname{Re}}
\newcommand{\tr}{\operatorname{Tr}}

\newcommand{\HC}{\BB H_{\BB C}}

\newcommand{\degt}{\widetilde{\operatorname{deg}}}

\usepackage[OT2,T1]{fontenc}
\newcommand\textcyr[1]{{\fontencoding{OT2}\fontfamily{wncyr}\selectfont #1}}
\newcommand{\Zh}{\textit{\textcyr{Zh}}}

\begin{document}

\title{\bf Magic Identities for the Conformal Four-Point Integrals;
the Minkowski Metric Case}
\author{Matvei Libine}
\maketitle

\begin{abstract}
The original ``magic identities'' are due to J.~M.~Drummond, J.~Henn,
V.~A.~Smirnov and E.~Sokatchev; they assert that all $n$-loop box integrals
for four scalar massless particles are equal to each other \cite{DHSS}.
The authors give a proof of the magic identities for the Euclidean metric case
only and claim that the result is also true in the Minkowski metric.
However, the Minkowski case is much more subtle and requires specification
of the relative positions of cycles of integration to make these identities
correct.
In this article we prove the magic identities in the Minkowski metric case
and, in particular, specify the cycles of integration.

Our proof of magic identities relies on previous results from \cite{L1,L2},
where we give a mathematical interpretation of the $n$-loop box integrals
in the context of representations of a Lie group $U(2,2)$ and quaternionic
analysis. The main result of \cite{L1,L2} is a (weaker) operator version
of the ``magic identities''.

No prior knowledge of physics or Feynman diagrams is assumed from the reader.
We provide a summary of all relevant results from quaternionic analysis to
make the article self-contained.
\end{abstract}



\section{Introduction}

This paper deals with planar conformal four-point integrals described by the
scalar box diagrams. They play an important role in physics, particularly
Yang-Mills conformal field theory (for more details see \cite{DHSS}
and references therein).
These diagrams have been thoroughly studied by physicists.
For example, the integral described by the one-loop Feynman diagram
is known to express the hyperbolic volume of an ideal tetrahedron
and is given by the dilogarithm function \cite{DD, W};
there are explicit expressions for the integrals described by the
ladder diagrams in terms of polylogarithms \cite{UD}.
Perhaps the most important property of the planar conformal four-point
integrals are the ``magic identities'' due to J.~M.~Drummond, J.~Henn,
V.~A.~Smirnov and E.~Sokatchev \cite{DHSS}.
These identities assert that all $n$-loop box integrals
for four scalar massless particles are equal to each other.
We will discuss these magic identities in Subsection \ref{magic-id-subsect}.

The original paper \cite{DHSS} gives a proof of the magic identities for
the Euclidean metric case only and claims that the result is also true in
the Minkowski metric.
In the Euclidean case, all variables belong to $\BB R^4$ or $\BB H$,
and there are no convergence issues whatsoever.
On the other hand, the Minkowski case -- which is the case we consider --
is much more subtle.
In order to deal with convergence issues, we must
consider the so-called ``off-shell Minkowski integrals'' or
perturb the cycles of integration inside $\BB H \otimes \BB C$.
Then the relative position of the cycles becomes very important.
In fact, choosing the ``wrong'' cycles typically results in the integral being
zero. The ``right'' choice of cycles was specified in \cite{L2}, we recall it
in Subsection \ref{BoxDiag_subsection}.

Our proof relies on previous results from \cite{L1, L2}, where we find a
representation-theoretic meaning of all planar conformal four-point integrals.
To each such integral, we associate an operator $L^{(n)}$ on
${\cal H}^+ \otimes {\cal H}^+$, where ${\cal H}^+$ denotes the space of
harmonic functions on the algebra of quaternions $\BB H$.
By Proposition 12 in \cite{L2} (restated here as Proposition \ref{Prop12}),
this operator $L^{(n)}$ is $\mathfrak{u}(2,2)$-equivariant and maps
${\cal H}^+ \otimes {\cal H}^+$ into itself.
We have a decomposition of $\mathfrak{u}(2,2)$-representations into distinct
irreducible components:
\begin{equation*}  
(\pi^0_l, {\cal H}^+) \otimes (\pi^0_r, {\cal H}^+) \simeq
\bigoplus_{k=1}^{\infty} \bigl( \rho_k, (\Zh^+_k)_{irr} \bigr).
\end{equation*}
Then, by Schur's Lemma, $L^{(n)}$ acts on each irreducible component
$\bigl( \rho_k, (\Zh^+_k)_{irr} \bigr)$ by multiplication by some scalar
$\mu^{(n)}_k$, and the main result of \cite{L2} is Theorem 14, which
describes these scalars.
Since these scalars depend only on the number of loops $n$ in the diagram
and not on the composition of the diagram,
we immediately obtain magic identities for the operators $L^{(n)}$:
Any two box diagrams with the same number of loops produce the same operator
$L^{(n)}$ on ${\cal H}^+ \otimes {\cal H}^+$ (Corollary 15 in \cite{L2}
which is restated here as Theorem \ref{operator-magic}).

We prove magic identities (Theorem \ref{magic}) by extending the operator
$L^{(n)}$ to a larger tensor product space $\Zh \otimes \Zh$,
this operator will be called $\bar L^{(n)}$.
The space $\Zh$ will be introduced in Subsection \ref{Zh-subsection};
it has a key property that the values of $\bar L^{(n)}$ on $\Zh \otimes \Zh$
uniquely determine the underlying $n$-loop four-point integral.
On the other hand, we will show that
$$
(\Zh \otimes \Zh)/ \ker \bar L^{(n)} \simeq {\cal H}^+ \otimes {\cal H}^+
$$
(Proposition \ref{harmonic-input+output}).
This allows us to reduce the original magic identities (Theorem \ref{magic})
to the already established operator version of magic identities for $L^{(n)}$
(Theorem 14 and Corollary 15 in \cite{L2}).

There are also other conformally invariant Feynman integrals,
such as Feynman integrals corresponding to non-planar graphs
(see, for example, \cite{DDEHPS} and references therein).
It appears that even the three- and four-loop non-planar integrals are rather
difficult to compute.
Methods developed in this paper could apply to these conformal integrals too,
and one could approach this type of Feynman integrals by describing associated 
$\mathfrak{u}(2,2)$-equivariant operators on an appropriate space of functions.

The paper is organized as follows.
In Section \ref{fd-section} we review the box diagrams and the corresponding
conformal four-point integrals, state the magic identities
(Theorem \ref{magic}) and specify the contours of integration.
In Section \ref{preliminaries} we establish our notations and state relevant
results from quaternionic analysis and \cite{L1, L2}.
In particular, we restate the operator version of magic identities
(Theorem \ref{operator-magic}).
In Section \ref{proof-section} we prove the original magic identities in the
Minkowski metric case (Theorem \ref{magic}).

\section{Conformal Four-Point Integrals and Magic Identities} \label{fd-section}

In this section we recall from \cite{L2} our description of
the planar conformal four-point integrals represented
by the $n$-loop box diagrams and the ``magic identities'' due to \cite{DHSS}
that assert that integrals represented by diagrams with the same number of
loops are, in fact, equal to each other.

\subsection{Some Notations}  \label{subsection_2.1}

We recall some notations from \cite{FL1}.
Let $\HC$ denote the space of complexified quaternions:
$\HC = \BB H \otimes \BB C$, it can be identified with the algebra of
$2 \times 2$ complex matrices:
$$
\HC = \BB H \otimes \BB C \simeq \biggl\{
Z = \begin{pmatrix} z_{11} & z_{12} \\ z_{21} & z_{22} \end{pmatrix}
; \: z_{ij} \in \BB C \biggr\}
= \biggl\{ Z= \begin{pmatrix} z^0-iz^3 & -iz^1-z^2 \\ -iz^1+z^2 & z^0+iz^3
\end{pmatrix} ; \: z^k \in \BB C \biggr\}.
$$
For $Z \in \HC$, we write
$$
N(Z) = \det \begin{pmatrix} z_{11} & z_{12} \\ z_{21} & z_{22} \end{pmatrix}
= z_{11}z_{22}-z_{12}z_{21} = (z^0)^2 + (z^1)^2 + (z^2)^2 + (z^3)^2
$$
and think of it as the norm of $Z$.
We realize $U(2)$ as
$$
U(2) = \{ Z \in \HC ;\: Z^*=Z^{-1} \},
$$
where $Z^*$ denotes the complex conjugate transpose of a complex matrix $Z$.
For $R>0$, we set
$$
U(2)_R = \{ RZ ;\: Z \in U(2) \} \quad \subset \HC
$$
and orient it as in \cite{FL1}, so that
$$
\int_{U(2)_R} \frac{dV}{N(Z)^2} = -2\pi^3 i,
$$
where $dV$ is a holomorphic 4-form
$$
dV = dz^0 \wedge dz^1 \wedge dz^2 \wedge dz^3
= \frac14 dz_{11} \wedge dz_{12} \wedge dz_{21} \wedge dz_{22}.
$$
The group $GL(2,\HC) \simeq GL(4,\BB C)$ acts on $\HC$ by fractional
linear (or conformal) transformations:
\begin{equation}  \label{conformal-action}
h: Z \mapsto (aZ+b)(cZ+d)^{-1} = (a'-Zc')^{-1}(-b'+Zd'),
\qquad Z \in \HC,
\end{equation}
where
$h = \bigl(\begin{smallmatrix} a & b \\ c & d \end{smallmatrix}\bigr)
\in GL(2,\HC)$ and 
$h^{-1} = \bigl(\begin{smallmatrix} a' & b' \\ c' & d' \end{smallmatrix}\bigr)$.

We often regard the group $U(2,2)$ as a subgroup of $GL(2,\HC)$,
as described in Subsection 3.5 of \cite{FL1}. That is
$$
U(2,2) = \Biggl\{ \begin{pmatrix} a & b \\ c & d \end{pmatrix}
\in GL(2,\HC) ;\: a,b,c,d \in \HC,\:
\begin{matrix} a^*a = 1+c^*c \\ d^*d = 1+b^*b \\ a^*b=c^*d \end{matrix}
\Biggr\}.
$$
The maximal compact subgroup of $U(2,2)$ is
\begin{equation*}  
U(2) \times U(2) = \biggl\{
\begin{pmatrix} a & 0 \\ 0 & d \end{pmatrix} \in GL(2, \HC);\:
a,d \in \HC, \: a^*a=d^*d=1 \biggr\}.
\end{equation*}
The action of the group $U(2,2)$ on $\HC$ by fractional linear transformations
(\ref{conformal-action}) preserves $U(2) \subset \HC$ and open domains
\begin{equation}  \label{D}
\BB D^+ = \{ Z \in \HC;\: ZZ^*<1 \}, \qquad
\BB D^- = \{ Z \in \HC;\: ZZ^*>1 \},
\end{equation}
where the inequalities $ZZ^*<1$ and $ZZ^*>1$ mean that the matrix $ZZ^*-1$
is negative and positive definite respectively.
The sets $\BB D^+$ and $\BB D^-$ both have $U(2)$ as the Shilov boundary.

Similarly, for each $R>0$, we can define a conjugate of $U(2,2)$
$$
U(2,2)_R = \begin{pmatrix} R & 0 \\ 0 & 1 \end{pmatrix} U(2,2)
\begin{pmatrix} R^{-1} & 0 \\ 0 & 1 \end{pmatrix} \quad \subset GL(2,\HC).
$$
Each group $U(2,2)_R$ is a real form of $GL(2,\HC)$, preserves $U(2)_R$
and open domains
\begin{equation}  \label{D_R}
\BB D^+_R = \{ Z \in \HC ;\: ZZ^*<R^2 \}, \qquad
\BB D^-_R = \{ Z \in \HC ;\: ZZ^*>R^2 \}.
\end{equation}
These sets $\BB D^+_R$ and $\BB D^-_R$ both have $U(2)_R$
as the Shilov boundary.

\subsection{Planar Conformal Four-Point Integrals}

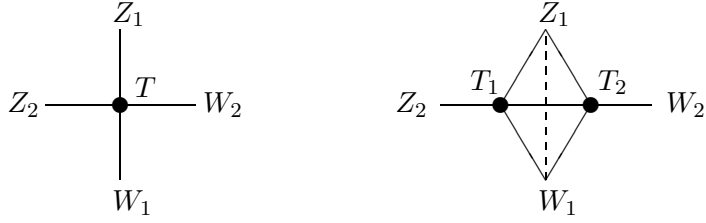
\begin{figure}
\begin{center}
\begin{subfigure}{0.3\textwidth}
\centering
\setlength{\unitlength}{1mm}
\begin{picture}(30,28)
\put(15,14){\circle*{2}}
\put(15,4){\line(0,1){20}}
\put(5,14){\line(1,0){20}}
\put(0,13){$Z_2$}
\put(14,25){$Z_1$}
\put(26,13){$W_2$}
\put(14,0){$W_1$}
\put(17,15){$T$}
\end{picture}
\end{subfigure}
\qquad
\begin{subfigure}{0.3\textwidth}
\centering
\setlength{\unitlength}{1mm}
\begin{picture}(40,28)
\put(14,14){\circle*{2}}
\put(26,14){\circle*{2}}
\put(6,14){\line(1,0){28}}
\put(14,14){\line(3,5){6}}
\put(14,14){\line(3,-5){6}}
\put(26,14){\line(-3,5){6}}
\put(26,14){\line(-3,-5){6}}
\multiput(20,4)(0,2){10}{\line(0,1){1}}
\put(0,13){$Z_2$}
\put(19,25){$Z_1$}
\put(36,13){$W_2$}
\put(19,0){$W_1$}
\put(27,16){$T_2$}
\put(10,16){$T_1$}
\end{picture}
\end{subfigure}
\end{center}
\caption{One-loop (left) and two-loop (right) box (or ladder) diagrams.}
\label{12ladder}
\end{figure}

In this subsection we explain how to construct the box diagrams and
the corresponding (planar) conformal four-point integrals.
As in \cite{DHSS}, we use the coordinate space variable notation
(as opposed to the momentum notation).
With this choice of variables, the one- and two-loop box (or ladder) diagrams
are represented as in Figure \ref{12ladder}.
The simplest conformal four-point integral is the one-loop box integral
\begin{equation}  \label{l^{(1)}}
l^{(1)}(Z_1,Z_2;W_1,W_2) =
\frac i{2\pi^3} \int_{T \in U(2)_r}
\frac{dV}{N(Z_1-T) \cdot N(Z_2-T) \cdot N(W_1-T) \cdot N(W_2-T)}.
\end{equation}
Here, $r>0$, $Z_1,Z_2 \in \BB D^-_r$ and $W_1,W_2 \in \BB D^+_r$.
Then we have the two-loop box integral
\begin{multline*}
-4\pi^6 \cdot l^{(2)}(Z_1,Z_2;W_1,W_2) \\
= \iint_{\genfrac{}{}{0pt}{}{T_1 \in U(2)_{r_1}}{T_2 \in U(2)_{r_2}}}
\frac{|Z_1-W_1|^2 \cdot |T_1-T_2|^{-2} \, dV_{T_1} \, dV_{T_2}}
{|Z_1-T_1|^2 \cdot |Z_2-T_1|^2 \cdot
|W_1-T_1|^2 \cdot |Z_1-T_2|^2 \cdot |W_1-T_2|^2 \cdot |W_2-T_2|^2},
\end{multline*}
where we write $|Z-W|^2$ for $N(Z-W)$ in order to fit the formula on page.
Here, $r_1>r_2>0$, $Z_1, Z_2 \in \BB D^-_{r_1}$,
$W_1, W_2 \in \BB D^+_{r_2}$.
The factor $|Z_1-W_1|^2=N(Z_1-W_1)$ in the numerator is not involved in
integration and gives $l^{(2)}$ desired conformal properties
(Lemma \ref{conformal}).

In general, one obtains the integral from the box diagram by building
a rational function by writing a factor
$$
\begin{cases}
N(Y_i-Y_j)^{-1} &
\text{if there is a solid edge joining variables $Y_i$ and $Y_j$};  \\
N(Y_i-Y_j) & \text{if there is a dashed edge joining variables $Y_i$ and $Y_j$},
\end{cases}
$$
and then integrating over the solid vertices.
The issue of contours of integration (and, in particular, their relative
position) will be addressed at the end of the next subsection.

\subsection{Box Diagrams}  \label{BoxDiag_subsection}

The box diagrams are obtained by starting with the one-loop box diagram
(Figure \ref{12ladder}) and attaching the so-called ``slingshots'',
as explained in \cite{DHSS}.
Figures \ref{slingshot+one-loop1} and \ref{slingshot+one-loop2}
show the two possible results of attaching a slingshot to the one-loop diagram;
these are called the two-loop box diagrams.
Then Figures \ref{slingshot+two-loop1} and \ref{slingshot+two-loop2}
show two different results of attaching a slingshot to the two-loop box
diagrams; these are called the three-loop box diagrams.
In general, if one has an $(n-1)$-loop box diagram $d^{(n-1)}$ -- that is a
box diagram  obtained by attaching $n-2$ slingshots to the one-loop box diagram
-- there are four ways of attaching a slingshot to form an $n$-loop box diagram
$d^{(n)}$: the hollow vertex of the slingshot can be attached to any of the
vertices labeled $Z_1$, $Z_2$, $W_1$ or $W_2$.
For example, Figure \ref{slingshot+gendiagram}
illustrates a slingshot with the hollow vertex being attached to
the vertex labeled $Z_2$, then the ends of the slingshot with the ``string''
are attached to the adjacent vertices $Z_1$ and $W_1$, the hollow vertex of
the slingshot becomes solid and gets relabeled $T_n$, finally,
the vertex at the tip of the ``handle'' of the slingshot is labeled $Z_2$.
The other three cases are similar.
While there are four ways to attach a slingshot to $d^{(n-1)}$,
some of the resulting diagrams may be the same, since we treat
all slingshots as identical.
(The variables $T_1,\dots,T_n$ get integrated out, so we treat the
diagrams obtained by permuting these variables as the same.)
Thus there are only two two-loop box diagrams, and they differ only by
rearranging labels $Z_1$, $Z_2$, $W_1$, $W_2$
(Figures \ref{slingshot+one-loop1} and \ref{slingshot+one-loop2}).
Figure \ref{ladder-var-label} shows a particular example of an $n$-loop
box diagram called the $n$-loop ladder diagram.
The reason for the ``box'', ``ladder'' and ``loop'' terminology becomes
apparent when one switches to the momentum variables, see Figure \ref{ladders},
and more figures are given in \cite{DHSS}.

\begin{figure}
\begin{center}
\begin{subfigure}{0.15\textwidth}
\centering
\includegraphics[scale=1]{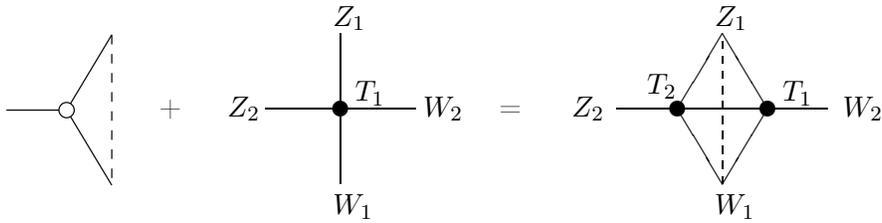}
\end{subfigure}
+
\begin{subfigure}{0.25\textwidth}
\centering
\setlength{\unitlength}{1mm}
\begin{picture}(30,28)
\put(15,14){\circle*{2}}
\put(15,4){\line(0,1){20}}
\put(5,14){\line(1,0){20}}
\put(0,13){$Z_2$}
\put(14,25){$Z_1$}
\put(26,13){$W_2$}
\put(14,0){$W_1$}
\put(17,15){$T_1$}
\end{picture}
\end{subfigure}
=
\begin{subfigure}{0.32\textwidth}
\centering
\setlength{\unitlength}{1mm}
\begin{picture}(40,28)
\put(14,14){\circle*{2}}
\put(26,14){\circle*{2}}
\put(6,14){\line(1,0){28}}
\put(14,14){\line(3,5){6}}
\put(14,14){\line(3,-5){6}}
\put(26,14){\line(-3,5){6}}
\put(26,14){\line(-3,-5){6}}
\multiput(20,4)(0,2){10}{\line(0,1){1}}
\put(0,13){$Z_2$}
\put(19,25){$Z_1$}
\put(36,13){$W_2$}
\put(19,0){$W_1$}
\put(28,15){$T_1$}
\put(10,16){$T_2$}
\end{picture}
\end{subfigure}
\end{center}
\caption{Attaching a slingshot to the one-loop box diagram.}
\label{slingshot+one-loop1}
\end{figure}

\begin{figure}
\begin{center}
\begin{subfigure}{0.15\textwidth}
\centering
\includegraphics[scale=1]{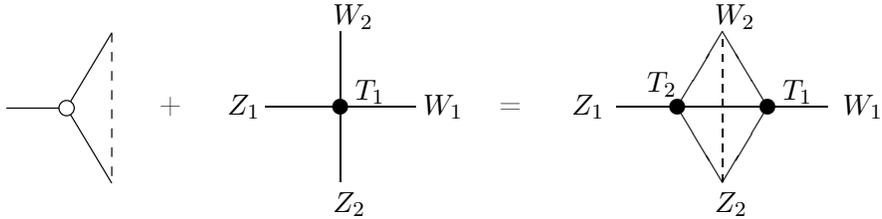}
\end{subfigure}
+
\begin{subfigure}{0.25\textwidth}
\centering
\setlength{\unitlength}{1mm}
\begin{picture}(30,28)
\put(15,14){\circle*{2}}
\put(15,4){\line(0,1){20}}
\put(5,14){\line(1,0){20}}
\put(0,13){$Z_1$}
\put(14,25){$W_2$}
\put(26,13){$W_1$}
\put(14,0){$Z_2$}
\put(17,15){$T_1$}
\end{picture}
\end{subfigure}
=
\begin{subfigure}{0.32\textwidth}
\centering
\setlength{\unitlength}{1mm}
\begin{picture}(40,28)
\put(14,14){\circle*{2}}
\put(26,14){\circle*{2}}
\put(6,14){\line(1,0){28}}
\put(14,14){\line(3,5){6}}
\put(14,14){\line(3,-5){6}}
\put(26,14){\line(-3,5){6}}
\put(26,14){\line(-3,-5){6}}
\multiput(20,4)(0,2){10}{\line(0,1){1}}
\put(0,13){$Z_1$}
\put(19,25){$W_2$}
\put(36,13){$W_1$}
\put(19,0){$Z_2$}
\put(28,15){$T_1$}
\put(10,16){$T_2$}
\end{picture}
\end{subfigure}
\end{center}
\caption{Another way of attaching a slingshot to the one-loop box diagram.}
\label{slingshot+one-loop2}
\end{figure}

\begin{figure}
\begin{center}
\begin{subfigure}{0.15\textwidth}
\centering
\includegraphics[scale=1]{slingshot.eps}
\end{subfigure}
+
\begin{subfigure}{0.3\textwidth}
\centering
\setlength{\unitlength}{1mm}
\begin{picture}(40,28)
\put(14,14){\circle*{2}}
\put(26,14){\circle*{2}}
\put(6,14){\line(1,0){28}}
\put(14,14){\line(3,5){6}}
\put(14,14){\line(3,-5){6}}
\put(26,14){\line(-3,5){6}}
\put(26,14){\line(-3,-5){6}}
\multiput(20,4)(0,2){10}{\line(0,1){1}}
\put(0,13){$Z_2$}
\put(19,25){$Z_1$}
\put(36,13){$W_2$}
\put(19,0){$W_1$}
\end{picture}
\end{subfigure}
=
\begin{subfigure}{0.32\textwidth}
\centering
\includegraphics[scale=1]{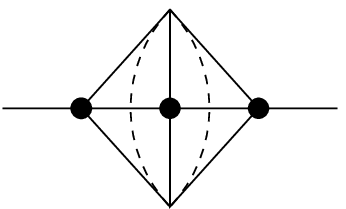}
\end{subfigure}
\hspace{-0.33\textwidth}
\begin{subfigure}{0.32\textwidth}
\centering
\setlength{\unitlength}{1mm}
\begin{picture}(46,28)
\put(0,13){$Z_2$}
\put(21,25){$Z_1$}
\put(41,13){$W_2$}
\put(21,0){$W_1$}
\end{picture}
\end{subfigure}
\end{center}
\caption{Attaching a slingshot to a two-loop box diagram.}
\label{slingshot+two-loop1}
\end{figure}

\begin{figure}
\begin{center}
\begin{subfigure}{0.15\textwidth}
\centering
\includegraphics[scale=1]{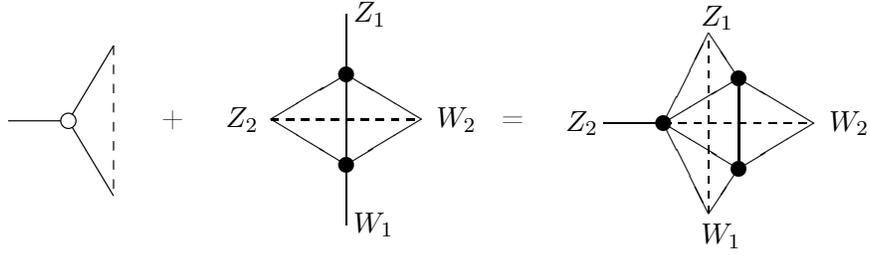}
\end{subfigure}
+
\begin{subfigure}{0.25\textwidth}
\centering
\setlength{\unitlength}{1mm}
\begin{picture}(31,30)
\put(16,9){\circle*{2}}
\put(16,21){\circle*{2}}
\put(16,1){\line(0,1){28}}
\put(6,15){\line(5,3){10}}
\put(6,15){\line(5,-3){10}}
\put(26,15){\line(-5,3){10}}
\put(26,15){\line(-5,-3){10}}
\multiput(6,15)(2,0){10}{\line(1,0){1}}
\put(0,14){$Z_2$}
\put(17,28){$Z_1$}
\put(28,14){$W_2$}
\put(17,0){$W_1$}
\end{picture}
\end{subfigure}
=
\begin{subfigure}{0.3\textwidth}
\centering
\setlength{\unitlength}{1mm}
\begin{picture}(39,33)
\put(23,10){\circle*{2}}
\put(23,22){\circle*{2}}
\put(13,16){\circle*{2}}
\put(5,16){\line(1,0){8}}
\put(23,10){\line(0,1){12}}
\put(13,16){\line(1,2){6}}
\put(23,22){\line(-2,3){4}}
\put(13,16){\line(1,-2){6}}
\put(23,10){\line(-2,-3){4}}
\put(13,16){\line(5,3){10}}
\put(13,16){\line(5,-3){10}}
\put(33,16){\line(-5,3){10}}
\put(33,16){\line(-5,-3){10}}
\multiput(13,16)(2,0){10}{\line(1,0){1}}
\multiput(19,4)(0,2){12}{\line(0,1){1}}
\put(0,15){$Z_2$}
\put(18,29){$Z_1$}
\put(35,15){$W_2$}
\put(18,0){$W_1$}
\end{picture}
\end{subfigure}
\end{center}
\caption{Another way of attaching a slingshot to a two-loop box diagram.}
\label{slingshot+two-loop2}
\end{figure}

\begin{figure}
\begin{center}
\begin{subfigure}{0.15\textwidth}
\centering
\includegraphics[scale=1]{slingshot.eps}
\end{subfigure}
+
\begin{subfigure}{0.35\textwidth}
\centering
\includegraphics[scale=1]{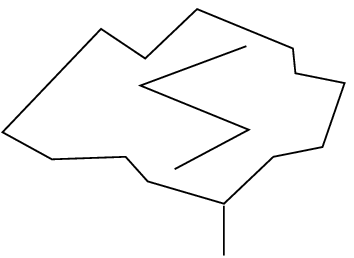}
\end{subfigure}
=
\begin{subfigure}{0.4\textwidth}
\centering
\includegraphics[scale=1]{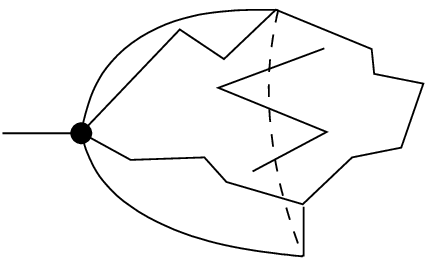}
\end{subfigure}
\end{center}
\vskip-33mm \hskip62mm $Z_1$ \hskip63.5mm $Z_1$

\vskip4mm \hskip77mm $W_2$ \hskip62.5mm $W_2$

\vskip-2mm \hskip106mm $T_n$

\vskip-2mm \hskip37mm $Z_2$ \hskip56mm $Z_2$

\vskip8mm \hskip65.5mm $W_1$ \hskip62.5mm $W_1$
\caption{Attaching a slingshot to a general box diagram.}
\label{slingshot+gendiagram}
\end{figure}

\begin{figure}
\begin{center}
\centerline{\includegraphics[scale=1.5]{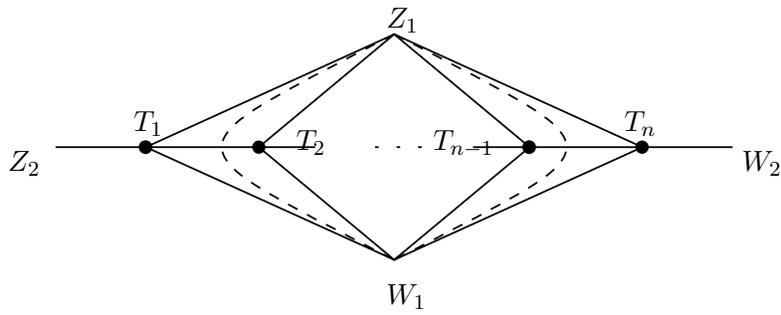}}
\vskip-36mm \hskip2mm $Z_1$

\vskip9mm $T_1$ \hskip60mm $T_n$

\vskip-4mm \:\:$T_2$ \hskip13mm $T_{n-1}$

\vskip-2mm $Z_2$ \hskip92mm $W_2$

\vskip13mm \hskip3mm $W_1$
\end{center}
\caption{$n$-loop ladder diagram (coordinate space variable).}
\label{ladder-var-label}
\end{figure}

\begin{figure}
\begin{center}
\begin{subfigure}{0.2\textwidth}
\centering
\includegraphics[scale=.8]{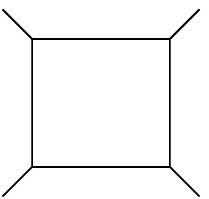}
\end{subfigure}
\begin{subfigure}{0.25\textwidth}
\centering
\includegraphics[scale=.8]{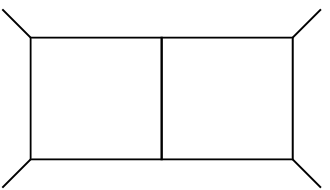}
\end{subfigure}
\begin{subfigure}{0.42\textwidth}
\centering
\includegraphics[scale=.8]{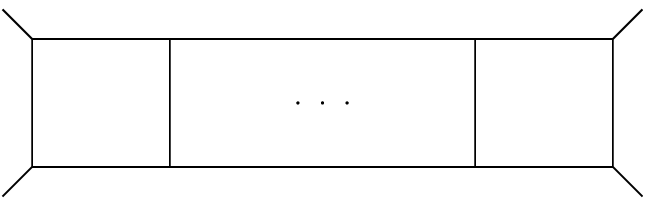}
\end{subfigure}
\end{center}
\caption{One-, two- and $n$-loop box or ladder diagrams in momentum variables.}
\label{ladders}
\end{figure}

In order to specify the cycles of integration, we introduce a partial ordering
on the variables in each $n$-loop box diagram $d^{(n)}$.
For the one-loop box diagram (Figure \ref{12ladder}) the relations are
$$
W_1,W_2 \prec T \prec Z_1,Z_2.
$$
Suppose that an $n$-loop box diagram $d^{(n)}$ is obtained
from an $(n-1)$-loop diagram $d^{(n-1)}$ by adding a slingshot.
Then $d^{(n)}$ will have one new relation for each solid edge of the slingshot,
plus those implied by the transitivity property.
Suppose, by induction, that the partial ordering for the variables in
$d^{(n-1)}$ are already specified.
We label the solid vertices in $d^{(n-1)}$ as $T_1, \dots, T_{n-1}$.
There are exactly four ways of attaching a slingshot to $d^{(n-1)}$
-- so that one of $Z_1$, $Z_2$, $W_1$ or $W_2$ becomes a solid vertex and gets
relabeled as $T_n$.
\begin{itemize}
\item
If $d^{(n)}$ is obtained from $d^{(n-1)}$ by adding the slingshot so
that $Z_1$ becomes a solid vertex, the relations in $d^{(n-1)}$ carry over
to $d^{(n)}$ with $Z_1$ replaced with $T_n$. Then we get new relations
$$
W_2 \prec T_n \prec Z_1, Z_2
$$
(plus those implied by the transitivity property).

\item
If $d^{(n)}$ is obtained from $d^{(n-1)}$ by adding the slingshot so
that $Z_2$ becomes a solid vertex, the relations in $d^{(n-1)}$ carry over
to $d^{(n)}$ with $Z_2$ replaced with $T_n$. Then we get new relations
$$
W_1 \prec T_n \prec Z_1, Z_2
$$
(plus those implied by the transitivity property).

\item
If $d^{(n)}$ is obtained from $d^{(n-1)}$ by adding the slingshot so
that $W_1$ becomes a solid vertex, the relations in $d^{(n-1)}$ carry over
to $d^{(n)}$ with $W_1$ replaced with $T_n$. Then we get new relations
$$
W_1, W_2 \prec T_n \prec Z_2
$$
(plus those implied by the transitivity property).

\item
If $d^{(n)}$ is obtained from $d^{(n-1)}$ by adding the slingshot so
that $W_2$ becomes a solid vertex, the relations in $d^{(n-1)}$ carry over
to $d^{(n)}$ with $W_2$ replaced with $T_n$. Then we get new relations
$$
W_1, W_2 \prec T_n \prec Z_1
$$
(plus those implied by the transitivity property).
\end{itemize}
This completely defines the partial ordering on the variables in $d^{(n)}$.
We choose real numbers $r_1,\dots,r_n>0$ such that $r_i < r_j$
whenever $T_i \prec T_j$ (it is easy to check that such a choice is
always possible). Then each $T_k$ gets integrated over $U(2)_{r_k}$.
Finally,
\begin{equation}  \label{r_max}
Z_i \in \BB D^-_{r_{\text{max},i}}, \quad \text{where }
r_{\text{max},i} = \max\{r_k ;\: T_k \prec Z_i \}, \qquad i=1,2;
\end{equation}
\begin{equation}  \label{r_min}
W_i \in \BB D^+_{r_{\text{min},i}}, \quad \text{where }
r_{\text{min},i} = \min\{r_k ;\: W_i \prec T_k \}, \qquad i=1,2.
\end{equation}

For future use, we make the following observations.
If $Y$ is one of the vertices of $d^{(n)}$ labeled $Z_1$, $Z_2$, $W_1$ or $W_2$,
then
\begin{equation}  \label{1edge}
\# \{\text{solid edges at $Y$} \} -
\# \{\text{dashed edges at $Y$} \} =1.
\end{equation}
Similarly, if $T$ is any solid vertex of the diagram $d^{(n)}$,
\begin{equation}  \label{4edges}
\# \{\text{solid edges at $T$} \} -
\# \{\text{dashed edges at $T$} \} =4,
\end{equation}
Both assertions can easily be shown by induction on the number of vertices in
$d^{(n)}$.

If desired, by Corollary 90 in \cite{FL1} the integrals over various $U(2)_r$'s
can be replaced by integrals over the Minkowski space $\BB M$ via an
appropriate ``Cayley transform''.
This means that these integrals are what the physicists call
``off-shell Minkowski integrals''.

\subsection{Magic Identities}  \label{magic-id-subsect}

In this subsection we state the so-called ``magic identities'' due to
J.~M.~Drummond, J.~Henn, V.~A.~Smirnov and E.~Sokatchev \cite{DHSS}.
Informally, they assert that all planar conformal four-point box integrals
obtained by adding the same number of slingshots to the one-loop integral
are equal. In other words, only the number of slingshots matters and
not how they are attached.

\begin{thm}  \label{magic}
Let $l^{(n)}(Z_1,Z_2;W_1,W_2)$ and $\tilde l^{(n)}(Z_1,Z_2;W_1,W_2)$ be two planar
conformal four-point integrals corresponding to any two $n$-loop box diagrams,
then
$$
l^{(n)}(Z_1,Z_2;W_1,W_2) = \tilde l^{(n)}(Z_1,Z_2;W_1,W_2).
$$
\end{thm}

In particular, we can parametrize the planar conformal four-point integrals by
the number of loops in the diagrams and choose a single representative from
the set of all $n$-loop diagrams, such as the $n$-loop ladder diagram
(Figures \ref{ladder-var-label} and \ref{ladders}).

The original paper \cite{DHSS} gives a proof for the Euclidean metric case and
claims that the result is also true in the Minkowski metric.
In the Euclidean case, the box integrals are produced by making all
variables belong to $\BB R^4 \simeq \BB H$ and replacing all cycles of
integration by $\BB H$.
Then $N(X-Y)$ is just the square of the Euclidean distance between $X$ and $Y$.
There are no convergence issues whatsoever.
On the other hand, the Minkowski case (which is the case we consider)
is much more subtle. In order to deal with convergence issues, we must
consider the so-called ``off-shell Minkowski integrals'' or make the cycles
of integration to be various $U(2)_r$'s as described above.
Then the relative position of cycles becomes very important;
it can be seen from the proof of Theorem 14 in \cite{L2} that choosing
the ``wrong'' cycles often results in integral being zero.

\section{Preliminaries}  \label{preliminaries}

In this section we establish further notations and state relevant results from
quaternionic analysis. We mostly follow our previous papers
\cite{FL1, FL2, L1, L2}.

\subsection{The Group $\HC^{\times}$ and Its Matrix Coefficients}  \label{matrix-coeff-subsection}

We denote by $\HC^{\times}$ the group of invertible complexified quaternions:
$$
\HC^{\times} = \{ Z \in \HC ;\: N(Z) \ne 0 \} \simeq GL(2,\BB C).
$$
Let $(\tau_{\frac12},\BB S)$ be the tautological 2-dimensional
representation of $\HC^{\times}$.
Then, for $l=0,\frac12,1,\frac32, \dots$, we denote by $(\tau_l,V_l)$ the
$2l$-th symmetric power product of $(\tau_{\frac12},\BB S)$.
(In particular, $(\tau_0,V_0)$ is the trivial one-dimensional representation.)
Thus, each $(\tau_l,V_l)$ is an irreducible representation of $\HC^{\times}$
of dimension $2l+1$.
A concrete realization of $(\tau_l,V_l)$ as well as an isomorphism
$V_l \simeq \BB C^{2l+1}$ suitable for our purposes are described in
Subsection 2.5 of \cite{FL1}.

Recall the matrix coefficient functions of $\tau_l(Z)$ described by
equation (27) of \cite{FL1} (cf. \cite{V}):
\begin{equation*}  
t^l_{n\,\underline{m}}(Z) = \frac 1{2\pi i}
\oint (sz_{11}+z_{21})^{l-m} (sz_{12}+z_{22})^{l+m} s^{-l+n} \,\frac{ds}s,
\qquad
\begin{matrix} l = 0, \frac12, 1, \frac32, \dots, \\ m,n \in \BB Z +l, \\
 -l \le m,n \le l, \end{matrix}
\end{equation*}
$Z=\bigl(\begin{smallmatrix} z_{11} & z_{12} \\
z_{21} & z_{22} \end{smallmatrix}\bigr) \in \HC$,
the integral is taken over a loop in $\BB C$ going once around the origin
in the counterclockwise direction.
We regard these functions as polynomials on $\HC$.
It is also useful to recall that
$$
t^l_{m\underline{n}}(Z^{-1}) 
\quad \text{is proportional to} \quad
N(Z)^{-2l} \cdot t^l_{-n\underline{-m}}(Z).
$$

A symmetric bilinear pairing on $\BB C$-valued functions on $\HC^{\times}$
is given by an expression
\begin{equation}  \label{pairing}
\langle f_1,f_2 \rangle =
\frac i{2\pi^3} \int_{Z \in U(2)_R} f_1(Z) \cdot f_2(Z) \,dV.
\end{equation}
We have the following orthogonality relations with respect to this pairing:
\begin{equation}  \label{orthogonality}
\bigl\langle N(Z)^{k'} \cdot t^{l'}_{n'\,\underline{m'}}(Z),
N(Z)^{-k-2} \cdot t^l_{m\underline{n}}(Z^{-1}) \bigr\rangle
= \frac1{2l+1} \delta_{kk'}\delta_{ll'} \delta_{mm'} \delta_{nn'},
\end{equation}
where the indices $k,l,m,n$ are
$l = 0, \frac12, 1, \frac32, \dots$, $m,n \in \BB Z +l$, $-l \le m,n \le l$,
$k \in \BB Z$ and similarly for $k',l',m',n'$.
The restrictions of the functions $N(Z)^k \cdot t^l_{n\,\underline{m}}(Z)$'s
to $U(2)$ form a complete orthogonal basis of functions on $U(2)$.
In other words, a continuous function on $U(2)$ that is orthogonal to all
$N(Z)^k \cdot t^l_{n\,\underline{m}}(Z)$'s with respect to the pairing
(\ref{pairing}) must be zero.
This follows immediately from the fact that the restrictions of the matrix
coefficients $t^l_{n\,\underline{m}}(Z)$'s to $SU(2) \subset U(2)$ form a
complete orthogonal basis of functions on $SU(2)$ (see \cite{V}).

We often use Proposition 25 from \cite{FL1}:

\begin{prop}
We have the following matrix coefficient expansion
\begin{equation}  \label{1/N-expansion}
\frac 1{N(Z-W)}= N(W)^{-1} \cdot \sum_{l,m,n}
t^l_{n\,\underline{m}}(Z) \cdot t^l_{m\,\underline{n}}(W^{-1}),
\end{equation}
which converges uniformly on compact subsets in the region
$\{ (Z,W) \in \HC \times \HC^{\times}; \: ZW^{-1} \in \BB D^+ \}$.
The sum is taken first over all $m,n = -l, -l+1, \dots, l$, then over
$l=0,\frac 12, 1, \frac 32,\dots$.
\end{prop}

\subsection{Harmonic Functions on $\HC^{\times}$}

As in Section 2 of \cite{FL2}, we consider the space $\widetilde{\cal H}$
consisting of $\BB C$-valued functions on $\HC$ (possibly with singularities)
that are holomorphic with respect to the complex variables
$z_{11},z_{12},z_{21},z_{22}$ and harmonic, i.e. annihilated by
$$
\square 
= 4\biggl( \frac{\partial^2}{\partial z_{11}\partial z_{22}}
- \frac{\partial^2}{\partial z_{12}\partial z_{21}} \biggr)
= \frac{\partial^2}{(\partial z^0)^2} + \frac{\partial^2}{(\partial z^1)^2}
+ \frac{\partial^2}{(\partial z^2)^2} + \frac{\partial^2}{(\partial z^3)^2}.
$$
Then the conformal group $GL(2,\HC)$ acts on $\widetilde{\cal H}$ by two
slightly different actions:
\begin{align*}
\pi^0_l(h): \: \phi(Z) \quad &\mapsto \quad \bigl( \pi^0_l(h)\phi \bigr)(Z) =
\frac 1{N(cZ+d)} \cdot \phi \bigl( (aZ+b)(cZ+d)^{-1} \bigr),  \\
\pi^0_r(h): \: \phi(Z) \quad &\mapsto \quad \bigl( \pi^0_r(h)\phi \bigr)(Z) =
\frac 1{N(a'-Zc')} \cdot \phi \bigl( (a'-Zc')^{-1}(-b'+Zd') \bigr),
\end{align*}
where
$h = \bigl(\begin{smallmatrix} a' & b' \\ c' & d' \end{smallmatrix}\bigr)
\in GL(2,\HC)$ and
$h^{-1} = \bigl(\begin{smallmatrix} a & b \\ c & d \end{smallmatrix}\bigr)$.
These two actions coincide on $SL(2,\HC) \simeq SL(4,\BB C)$
which is defined as the connected Lie subgroup of $GL(2,\HC)$ with Lie algebra
$$
\mathfrak{sl}(2,\HC) = \{ x \in \mathfrak{gl}(2,\HC) ;\: \re (\tr x) =0 \}
\simeq \mathfrak{sl}(4,\BB C).
$$

Recall harmonic polynomial functions on $\HC^{\times}$:
\begin{align*}
  {\cal H}^+ &=   \bigl\{ \phi \in \BB C[z_{11},z_{12},z_{21},z_{22}] ;\:
  \square \phi =0 \bigr\}  \\
  &= \BB C\text{-span of } \bigl\{ t^l_{n \, \underline{m}}(Z) \bigr\},  \\
  {\cal H}^- &= \bigl\{ \phi \in \BB C[z_{11},z_{12},z_{21},z_{22}, N(Z)^{-1}] ;\:
  N(Z)^{-1} \cdot \phi(Z^{-1}) \in {\cal H^+} \bigr\}  \\
  &= \BB C\text{-span of } \bigl\{
  N(Z)^{-1} \cdot t^l_{m \, \underline{n}}(Z^{-1}) \bigr\},  \\
  {\cal H} &= \bigl\{ \phi \in \BB C[z_{11},z_{12},z_{21},z_{22}, N(Z)^{-1}] ;\:
  \square \phi =0 \bigr\}  \\
  &= {\cal H}^+ \oplus {\cal H}^-,
\end{align*}
where $l = 0, \frac12, 1, \frac32, \dots$, $m,n \in \BB Z +l$,
$-l-1 \le m \le l+1$, $-l \le n \le l$.
Differentiating the actions $\pi^0_l$ and $\pi^0_r$, we obtain actions of
$\mathfrak{gl}(2,\HC) \simeq \mathfrak{gl}(4,\BB C)$ which preserve
the spaces ${\cal H}$, ${\cal H}^-$ and ${\cal H}^+$.
By abuse of notation, we denote these Lie algebra actions by
$\pi^0_l$ and $\pi^0_r$ respectively.
They are described in Subsection 3.2 of \cite{FL2}.

We have a non-degenerate antisymmetric $\mathfrak{gl}(2,\HC)$-invariant
bilinear pairing between $(\pi^0_l, {\cal H})$ and $(\pi^0_r, {\cal H})$
given by equation (19) in \cite{L2}:
\begin{equation}  \label{H-pairing2}
(\phi_1,\phi_2) =
\frac i{2\pi^3} \int_{Z \in U(2)_R} (\degt\phi_1)(Z) \cdot \phi_2(Z)
\,\frac{dV}{N(Z)}, \qquad \phi_1,\phi_2 \in {\cal H}.
\end{equation}
(This pairing is independent of the choice of $R>0$.)
The operator $\degt$ is the degree operator plus identity:
$$
\degt f = f + \deg f = f + z_{11}\frac{\partial f}{\partial z_{11}} +
z_{12}\frac{\partial f}{\partial z_{12}}
+ z_{21}\frac{\partial f}{\partial z_{21}}
+ z_{22}\frac{\partial f}{\partial z_{22}}.
$$
Pairing (\ref{H-pairing2}) satisfies the following symmetry relation. Let
$$
\tilde \phi_i(Z) = N(Z)^{-1} \cdot \phi_i(Z^{-1}), \qquad i=1,2;
$$
then
\begin{equation}  \label{inversion-symmetry}
(\phi_1,\phi_2) = -(\tilde\phi_1,\tilde\phi_2).
\end{equation}
(This follows from Lemma 61 in \cite{FL1}.)

\subsection{Representations $(\varpi_2^l,\Zh)$, $(\varpi_2^r,\Zh)$}  \label{Zh-subsection}

Let $\widetilde{\Zh}$ denote the space of $\BB C$-valued functions on $\HC$
(possibly with singularities) which are holomorphic with respect to the
complex variables $z_{11}$, $z_{12}$, $z_{21}$, $z_{22}$.
We define two very similar actions of $GL(2,\HC)$ on $\widetilde{\Zh}$:
\begin{align*}
\varpi_2^l(h): \: f(Z) \quad &\mapsto \quad \bigl( \varpi_2^l(h)f \bigr)(Z) =
\frac {f \bigl( (aZ+b)(cZ+d)^{-1} \bigr)}{N(cZ+d)^2 \cdot N(a'-Zc')},  \\
\varpi_2^r(h): \: f(Z) \quad &\mapsto \quad \bigl( \varpi_2^r(h)f \bigr)(Z) =
\frac {f \bigl( (aZ+b)(cZ+d)^{-1} \bigr)}{N(cZ+d) \cdot N(a'-Zc')^2},
\end{align*}
where
$h = \bigl(\begin{smallmatrix} a' & b' \\ c' & d' \end{smallmatrix}\bigr)
\in GL(2,\HC)$ and 
$h^{-1} = \bigl(\begin{smallmatrix} a & b \\ c & d \end{smallmatrix}\bigr)$.
(These actions coincide on $SL(2,\HC)$.)
Note that $\varpi_2^l$ is the action $\varpi_2$ in the notations of \cite{L1}.

Recall polynomial functions on $\HC^{\times}$:
\begin{align*}
  \Zh^+ &= \{\text{polynomial functions on $\HC$}\}
  = \BB C[z_{11},z_{12},z_{21},z_{22}]  \\
  &= \BB C\text{-span of } \bigl\{ N(Z)^k \cdot t^l_{n \, \underline{m}}(Z);\:
  k=0,1,2,\dots \bigr\},  \\
  \Zh &= \bigl\{\text{polynomial functions on $\HC^{\times}$}\bigr\}
  = \BB C[z_{11},z_{12},z_{21},z_{22}, N(Z)^{-1}]  \\
  &= \BB C\text{-span of } \bigl\{ N(Z)^k \cdot t^l_{n \, \underline{m}}(Z);\:
  k \in \BB Z \bigr\},
\end{align*}
where $l = 0, \frac12, 1, \frac32, \dots$, $m,n \in \BB Z +l$,
$-l-1 \le m \le l+1$, $-l \le n \le l$.

Differentiating the $\varpi_2^l$ and $\varpi_2^r$ actions, we obtain actions
of $\mathfrak{gl}(2,\HC)$ (still denoted by $\varpi_2^l$ and $\varpi_2^r$
respectively), which preserve the spaces $\Zh$ and $\Zh^+$.

\begin{thm}[Theorem 8 in \cite{L1}]
The spaces
\begin{align*}
\Zh^+ &= \BB C \text{-span of }
\bigl\{ N(Z)^k \cdot t^l_{n\,\underline{m}}(Z);\: k \ge 0 \bigr\}, \\
\Zh_2^- &= \BB C \text{-span of }
\bigl\{ N(Z)^k \cdot t^l_{n\,\underline{m}}(Z);\: k \le -(2l+3) \bigr\}, \\
I_2^- &= \BB C \text{-span of }
\bigl\{ N(Z)^k \cdot t^l_{n\,\underline{m}}(Z);\: k \le -2 \bigr\}, \\
I_2^+ &= \BB C \text{-span of }
\bigl\{ N(Z)^k \cdot t^l_{n\,\underline{m}}(Z);\: k \ge -(2l+1) \bigr\}, \\
J_2 &= \BB C \text{-span of }
\bigl\{ N(Z)^k \cdot t^l_{n\,\underline{m}}(Z);\: -(2l+1) \le k \le -2 \bigr\}
\end{align*}
and their sums are the only proper subspaces of $\Zh$ that are invariant
under either $\varpi_2^l$ or $\varpi_2^r$ actions of $\mathfrak{gl}(2,\HC)$
(see Figure \ref{decomposition-fig2}).

The irreducible components of $(\varpi_2^l, \Zh)$ and $(\varpi_2^r, \Zh)$
are the subrepresentations
$$
(\varpi_2^*, \Zh^+), \qquad (\varpi_2^*, \Zh_2^-), \qquad (\varpi_2^*, J_2)
$$
and the quotients
$$
\bigl( \varpi_2^*, \Zh/(I_2^- \oplus \Zh^+) \bigr)
= \bigl( \varpi_2^*, I_2^+/(\Zh^+ \oplus J_2) \bigr), \quad
\bigl( \varpi_2^*, \Zh/(\Zh_2^- \oplus I_2^+) \bigr)
= \bigl( \varpi_2^*, I_2^-/(\Zh_2^- \oplus J_2) \bigr),
$$
where $*$ stands for $l$ or $r$.
\end{thm}

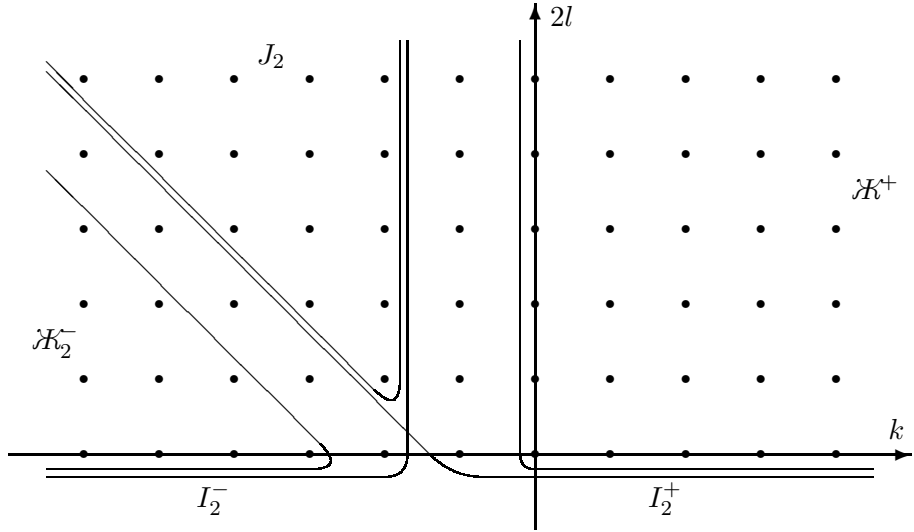
\begin{figure}
\begin{center}
\setlength{\unitlength}{1mm}
\begin{picture}(120,70)
\multiput(10,10)(10,0){11}{\circle*{1}}
\multiput(10,20)(10,0){11}{\circle*{1}}
\multiput(10,30)(10,0){11}{\circle*{1}}
\multiput(10,40)(10,0){11}{\circle*{1}}
\multiput(10,50)(10,0){11}{\circle*{1}}
\multiput(10,60)(10,0){11}{\circle*{1}}

\thicklines
\put(70,0){\vector(0,1){70}}
\put(0,10){\vector(1,0){120}}

\thinlines
\put(68,10){\line(0,1){55}}
\put(70,8){\line(1,0){45}}
\qbezier(68,10)(68,8)(70,8)

\put(52,20){\line(0,1){45}}
\put(48.6,18.6){\line(-1,1){43.6}}
\qbezier(52,20)(52,15.2)(48.6,18.6)

\put(5,8){\line(1,0){35}}
\put(41.4,11.4){\line(-1,1){36.4}}
\qbezier(40,8)(44.8,8)(41.4,11.4)

\put(63,7){\line(1,0){52}}
\put(56,10){\line(-1,1){51}}
\qbezier(63,7)(59,7)(56,10)

\put(5,7){\line(1,0){45}}
\put(53,10){\line(0,1){55}}
\qbezier(50,7)(53,7)(53,10)

\put(72,67){$2l$}
\put(117,12){$k$}
\put(3,24){$\Zh_2^-$}
\put(33,62){$J_2$}
\put(112,44){$\Zh^+$}
\put(25,3){$I_2^-$}
\put(85,3){$I_2^+$}
\end{picture}
\end{center}
\caption{Decomposition of $(\varpi_2^l,\Zh)$ and $(\varpi_2^r,\Zh)$
into irreducible components.}
\label{decomposition-fig2}
\end{figure}

The quotient representations can be identified as follows:

\begin{prop}[Proposition 10 in \cite{L1}]  \label{quotient-prop}
As representations of $\mathfrak{gl}(2,\HC)$,
$$
\bigl( \varpi_2^l, \Zh/(I_2^- \oplus \Zh^+) \bigr) \simeq (\pi^0_l, {\cal H}^+),
\qquad 
\bigl(\varpi_2^l, \Zh/(\Zh_2^- \oplus I_2^+)\bigr) \simeq (\pi^0_l, {\cal H}^-),
$$
$$
\bigl( \varpi_2^r, \Zh/(I_2^- \oplus \Zh^+) \bigr) \simeq (\pi^0_r, {\cal H}^+),
\qquad
\bigl(\varpi_2^r, \Zh/(\Zh_2^- \oplus I_2^+)\bigr) \simeq (\pi^0_r, {\cal H}^-),
$$
in all cases the isomorphism map being
\begin{equation}  \label{harm-iso}
{\cal H}^{\pm} \ni \phi(Z) \quad \mapsto \quad
\frac{\degt \phi(Z)}{N(Z)} \in
\begin{matrix} \Zh/(\Zh_2^- \oplus I_2^+) \\ \text{or} \\
\Zh/(I_2^- \oplus \Zh^+). \end{matrix}
\end{equation}
The inverse of this isomorphism is given by
$$
\begin{matrix} \Zh/(\Zh_2^- \oplus I_2^+) \\ \text{or} \\
\Zh/(I_2^- \oplus \Zh^+) \end{matrix} \ni f(Z)
\quad \mapsto \quad \biggl\langle f(Z), \frac1{N(Z-W)} \biggr\rangle_Z
= \frac i{2\pi^3} \int_{Z \in U(2)_R} \frac{f(Z)\,dV}{N(Z-W)} \in {\cal H}.
$$
\end{prop}

\subsection{Operator Version of the Magic Identities}

In \cite{L1,L2} we used the bilinear pairing (\ref{H-pairing2}) to
obtain integral operators $L^{(n)}$ on
$(\pi^0_l, {\cal H}^+) \otimes (\pi^0_r, {\cal H}^+)$
that have the conformal integrals $l^{(n)}$ as their kernels:
\begin{multline*}
L^{(n)} (\phi_1 \otimes \phi_2)(W_1,W_2) \\
= \Bigl(\frac{i}{2\pi^3}\Bigr)^2
\iint_{\genfrac{}{}{0pt}{}{Z_1 \in U(2)_{R_1}}{Z_2 \in U(2)_{R_2}}}
l^{(n)}(Z_1,Z_2;W_1,W_2) \cdot (\degt_{Z_1} \phi_1)(Z_1)
\cdot (\degt_{Z_2} \phi_2)(Z_2) \,\frac{dV_1}{N(Z_1)} \frac{dV_2}{N(Z_2)},
\end{multline*}
where $\phi_1, \phi_2 \in {\cal H}^+$, $R_1 > r_{\text{max},1}$,
$R_2 > r_{\text{max},2}$, $W_1 \in \BB D^+_{r_{\text{min},1}}$,
$W_2 \in \BB D^+_{r_{\text{min},2}}$
(recall that $r_{\text{max},i}$ and $r_{\text{min},i}$ are defined in
(\ref{r_max}) and (\ref{r_min})).
We restate Proposition 12 from \cite{L2}:

\begin{prop}  \label{Prop12}
For each $\phi_1, \phi_2 \in {\cal H}^+$, the function
$L^{(n)} (\phi_1 \otimes \phi_2)(W_1,W_2)$ is polynomial and harmonic
in each variable. In other words,
$L^{(n)} (\phi_1 \otimes \phi_2)(W_1,W_2) \in {\cal H}^+ \otimes {\cal H}^+$.
Moreover, the operator
$$
L^{(n)}: (\pi^0_l, {\cal H}^+) \otimes (\pi^0_r, {\cal H}^+)
\to (\pi^0_l, {\cal H}^+) \otimes (\pi^0_r, {\cal H}^+)
$$
is $\mathfrak{gl}(2,\HC)$-equivariant.
\end{prop}

The main result of \cite{L2} is Theorem 14 describing the action of the
operator $L^{(n)}$ on each irreducible component of
$(\pi^0_l, {\cal H}^+) \otimes (\pi^0_r, {\cal H}^+)$.
Since these actions of $L^{(n)}$ on the irreducible components depend only
on the number of loops $n$ in the diagram and not on the composition of
the diagram, we immediately obtain an operator version of magic identities:

\begin{thm} [Corollary 15 in \cite{L2}]  \label{operator-magic}
Let $L^{(n)}$ and $\tilde L^{(n)}$ be two integral operators
corresponding to any two $n$-loop box diagrams, then
$L^{(n)} = \tilde L^{(n)}$, as operators on ${\cal H}^+ \otimes {\cal H}^+$.
\end{thm}

The proof of Theorem 14 in \cite{L2} relies on the decomposition of the
tensor product representation
$(\pi^0_l, {\cal H}^+) \otimes (\pi^0_r, {\cal H}^+)$ of $\mathfrak{gl}(2,\HC)$
into irreducible components due to \cite{JV}.
Unfortunately, Theorem 82 in \cite{FL1}
(also restated as Theorem 5.4 in \cite{L1} and Theorem 13 in \cite{L2})
is not correctly stated and should be corrected as follows.

\begin{thm}
The image of the intertwining map $M_k$ from Theorem 85 in \cite{FL1} is an
irreducible subrepresentation of $(\rho_k,\Zh_k^+)$, $k=1,2,3,\dots$.

Let us denote this image by $(\Zh_k^+)_{irr}$.
The irreducible representations $\bigl( \rho_k, (\Zh_k^+)_{irr} \bigr)$,
$k=1,2,3,\dots$, of $\mathfrak{sl}(2,\HC)$ are pairwise non-isomorphic
and possess inner products which make them unitary representations of
the real form $\mathfrak{su}(2,2)$ of $\mathfrak{sl}(2,\HC)$.
\end{thm}

When $k=1$, we have $(\Zh_1^+)_{irr} = \Zh_1^+ = \Zh^+$.
Then equation (61) in \cite{FL1} (also restated as equations (34) in \cite{L1}
and (23) in \cite{L2}) should read as follows.
\begin{equation} \label{tensor-decomp}
(\pi^0_l, {\cal H}^+) \otimes (\pi^0_r, {\cal H}^+) \simeq
\bigoplus_{k=1}^{\infty} \bigl( \rho_k,(\Zh_k^+)_{irr} \bigr),
\end{equation}
This decomposition is obtained by treating ${\cal H}^+ \otimes {\cal H}^+$
as functions of two variables $Z, Z' \in \HC$ and filtering them by the
degree of vanishing on the diagonal $\HC \subset \HC \times \HC$. Then
$$
(\rho_1,\Zh^+) \quad \text{is generated by} \quad 1 \otimes 1,
$$
$$
\bigl( \rho_k, (\Zh_k^+)_{irr} \bigr) \quad \text{is generated by} \quad
(z_{ij}-z'_{ij})^{k-1} , \qquad k \ge 2.
$$
While decomposition (\ref{tensor-decomp}) is essential for \cite{L1} and
\cite{L2}, the arguments given there only require the existence of such
decomposition and an explicit choice of generators of 
each $(\rho_k,(\Zh_k^+)_{irr})$.
Thus the results of \cite{L1, L2} remain correct.

We also state an important conformal property of four-point box integrals.

\begin{lem} [Lemma 17 in \cite{L2}] \label{conformal}
For each $h = \bigl(\begin{smallmatrix} a' & b' \\
c' & d' \end{smallmatrix}\bigr) \in GL(2,\HC)$
sufficiently close to the identity, we have:
\begin{multline*}
l^{(n)}(\tilde Z_1, \tilde Z_2; \tilde W_1, \tilde W_2) \\
= N(a'-Z_1c') \cdot N(cZ_2+d) \cdot N(cW_1+d) \cdot N(a'-W_2c') \cdot
l^{(n)}(Z_1,Z_2;W_1,W_2),
\end{multline*}
where $h^{-1} = \bigl(\begin{smallmatrix} a & b \\
c & d \end{smallmatrix}\bigr)$,
$\tilde Z_i = (aZ_i+b)(cZ_i+d)^{-1}$ and $\tilde W_i = (aW_i+b)(cW_i+d)^{-1}$,
$i=1,2$.
\end{lem}

\section{Proof of Magic Identities}  \label{proof-section}

In this section we prove the Minkowski metric magic identities,
as stated in Theorem \ref{magic}.
The proof relies on the already established operator version of magic
identities (Theorem \ref{operator-magic}).

\subsection{Operator $\bar L^{(n)}$ and Its Properties}

We mostly work with integral operators
$\bar L^{(n)}$ on $(\varpi_2^l, \Zh) \otimes (\varpi_2^r, \Zh)$
obtained by pairing $l^{(n)}$ with elements of $\Zh \otimes \Zh$
via (\ref{pairing}):
\begin{equation}  \label{L-bar}
\bar L^{(n)} (f_1 \otimes f_2)(W_1,W_2)
= \Bigl(\frac{i}{2\pi^3}\Bigr)^2
\iint_{\genfrac{}{}{0pt}{}{Z_1 \in U(2)_{R_1}}{Z_2 \in U(2)_{R_2}}}
l^{(n)}(Z_1,Z_2;W_1,W_2) \cdot f_1(Z_1) \cdot f_2(Z_2) \,dV_1 dV_2,
\end{equation}
where $f_1, f_2 \in \Zh$, $R_1 > r_{\text{max},1}$,
$R_2 > r_{\text{max},2}$, $W_1 \in \BB D^+_{r_{\text{min},1}}$,
$W_2 \in \BB D^+_{r_{\text{min},2}}$.
Clearly, $L^{(n)}$ and $\bar L^{(n)}$ are related as
\begin{equation}  \label{L-L-bar-rel}
L^{(n)} (\phi_1 \otimes \phi_2)(W_1,W_2) = 
\bar L^{(n)} \biggl( \frac{\degt \phi_1(Z_1)}{N(Z_1)} \otimes
\frac{\degt \phi_2(Z_2)}{N(Z_2)} \biggr)(W_1,W_2), \qquad
\phi_1, \phi_2 \in {\cal H}^+,
\end{equation}
and the map
$$
{\cal H}^+ \ni \: \phi \mapsto \frac{\degt \phi(Z)}{N(Z)} \: \in \Zh
$$
is essentially the isomorphism (\ref{harm-iso}).
We will see shortly that $\bar L^{(n)}$ is a
$\mathfrak{gl}(2,\HC)$-equivariant map
$$
(\varpi_2^l, \Zh) \otimes (\varpi_2^r, \Zh) \to
(\pi^0_l, \Zh^+) \otimes (\pi^0_r, \Zh^+)
$$
(Proposition \ref{polynomial-prop}).
But as a preliminary definition let
$$
\widetilde{\Zh^+ \otimes \Zh^+}
= \bigl\{ \text{analytic functions $f(W_1,W_2):
  \BB D^+_{r_{\text{min},1}} \times \BB D^+_{r_{\text{min},2}} \to \BB C$} \bigr\}.
$$
By construction (\ref{L-bar}), $\bar L^{(n)}$ maps
$\Zh \otimes \Zh$ into $\widetilde{\Zh^+ \otimes \Zh^+}$.
We let $\mathfrak{gl}(2,\HC)$ act on $\widetilde{\Zh^+ \otimes \Zh^+}$ by
differentiating
$$
\bigl( (\pi^0_l \otimes \pi^0_r)(h) F \bigr)(W_1,W_2) =
\frac{F\bigl( (aW_1+b)(cW_1+d)^{-1}, (a'-W_2c')^{-1}(-b'+W_2d') \bigr)}
{N(cW_1+d) \cdot N(a'-W_2c')},
$$
where $F \in \widetilde{\Zh^+ \otimes \Zh^+}$,
$h = \bigl(\begin{smallmatrix} a' & b' \\ c' & d' \end{smallmatrix}\bigr)
\in GL(2,\HC)$ and
$h^{-1} = \bigl(\begin{smallmatrix} a & b \\ c & d \end{smallmatrix}\bigr)$.
Similarly, we obtain an action of the group $U(2) \times U(2)$
on $\widetilde{\Zh^+ \otimes \Zh^+}$.

\begin{lem}  \label{L-bar_equivariance}
The operator
$$
\bar L^{(n)}: (\varpi_2^l, \Zh) \otimes (\varpi_2^r, \Zh) \to
\bigl(\pi^0_l \otimes \pi^0_r, \widetilde{\Zh^+ \otimes \Zh^+} \bigr)
$$
is $U(2) \times U(2)$ and $\mathfrak{gl}(2,\HC)$-equivariant.
\end{lem}

\begin{proof}
Note that the integrand in (\ref{L-bar}) is a holomorphic form of highest
degree, hence independent of the choice of $R_1$ and $R_2$
as long as $R_1 > r_{\text{max},1}$ and $R_2 > r_{\text{max},2}$.
Thus, without loss of generality we may assume that $R_1=R_2=R$ for
some $R > r_{\text{max},1}, r_{\text{max},2}$.
Recall from Subsection \ref{subsection_2.1} that the group $U(2,2)_R$ is a
conjugate of $U(2,2)$, which is a real form of $GL(2,\HC)$ preserving $U(2)_R$,
$\BB D_R^+$ and $\BB D_R^-$.
We need to show that, for all $h \in U(2,2)_R$, the operator $\bar L^{(n)}$
defined by (\ref{L-bar}) commutes with the action of $h$. Writing
$h= \bigl(\begin{smallmatrix} a' & b' \\ c' & d' \end{smallmatrix}\bigr)$,
$h^{-1}= \bigl(\begin{smallmatrix} a & b \\ c & d \end{smallmatrix}\bigr)$,
$$
\tilde Z_i = (aZ_i+b)(cZ_i+d)^{-1}, \qquad
\tilde W_i = (aW_i+b)(cW_i+d)^{-1}, \qquad i=1,2,
$$
and using Lemma \ref{conformal} together with Lemma 61 from \cite{FL1},
we obtain:
\begin{multline*}
\iint_{\genfrac{}{}{0pt}{}{Z_1 \in U(2)_R}{Z_2 \in U(2)_R}}
l^{(n)}(Z_1,Z_2;W_1,W_2) \cdot
\bigl((\varpi_2^l\otimes \varpi_2^r)(h)F\bigr)(Z_1,Z_2) \,dV_1 dV_2  \\
= \iint_{\genfrac{}{}{0pt}{}{Z_1 \in U(2)_R}{Z_2 \in U(2)_R}}
\frac{l^{(n)}(Z_1,Z_2;W_1,W_2) \cdot F(\tilde Z_1,\tilde Z_2)}
{N(cZ_1+d)^2 \cdot N(a'-Z_1c') \cdot N(cZ_2+d) \cdot N(a'-Z_2c')^2}\,dV_1 dV_2 \\
= \iint_{\genfrac{}{}{0pt}{}{Z_1 \in U(2)_R}{Z_2 \in U(2)_R}}
\frac{N(cW_1+d)^{-1} \cdot N(a'-W_2c')^{-1} \cdot
l^{(n)}(\tilde Z_1,\tilde Z_2;\tilde W_1,\tilde W_2) \cdot
F(\tilde Z_1,\tilde Z_2)} {N(cZ_1+d)^2 \cdot N(a'-Z_1c')^2 \cdot N(cZ_2+d)^2
\cdot N(a'-Z_2c')^2} \,dV_1 dV_2 \\
= \frac1{N(cW_1+d) \cdot N(a'-W_2c')}
\iint_{\genfrac{}{}{0pt}{}{\tilde Z_1 \in U(2)_R}{\tilde Z_2 \in U(2)_R}}
l^{(n)}(\tilde Z_1,\tilde Z_2;\tilde W_1,\tilde W_2) \cdot
F(\tilde Z_1,\tilde Z_2) \,dV_1 dV_2,
\end{multline*}
where $F(Z_1,Z_2) \in \Zh \otimes \Zh$.
This proves the $U(2,2)_R$ and $U(2) \times U(2)$-equivariance.
The $\mathfrak{gl}(2,\HC)$-equivariance then follows since
$\mathfrak{gl}(2,\HC) \simeq \BB C \otimes \mathfrak{u}(2,2)_R$.
\end{proof}

\begin{lem}  \label{I^--annihilated-lem}
  The operator $\bar L^{(n)}$ annihilates $I_2^- \otimes \Zh$ and
  $\Zh \otimes I_2^-$.
\end{lem}

\begin{proof}
  Since $N(Z)^{-2}$ generates $I_2^-$, the families of functions
  \begin{equation}  \label{2-families-gen}
  \{ N(Z_1)^{-2} \otimes f_2(Z_2) ;\: f_2 \in \Zh \}
  \quad \text{and} \quad
  \{ f_1(Z_1) \otimes N(Z_2)^{-2} ;\: f_1 \in \Zh \}
  \end{equation}
  generate $I_2^- \otimes \Zh$ and  $\Zh \otimes I_2^-$ respectively as
  representations of $\mathfrak{gl}(2,\HC)$.
  Hence it is sufficient to show that $\bar L^{(n)}$ annihilates these
  generators.

  Consider the vertex of the diagram labeled $Z_1$.
  Suppose that there are $k$ solid edges at $Z_1$, then, by (\ref{1edge}),
  there are $k-1$ dashed edges at $Z_1$.
  When we evaluate $\bar L^{(n)} (N(Z_1)^{-2} \otimes f_2)(W_1,W_2)$,
  we can perform integration over $Z_1$ first. This means, when we
  integrate
  $$
  \frac{i}{2\pi^3} \int_{Z_1 \in U(2)_{R_1}}
  l^{(n)}(Z_1,Z_2;W_1,W_2) \cdot N(Z_1)^{-2} \cdot f_2(Z_2) \,dV_1,
  $$
  we are really integrating  
  \begin{equation}  \label{L-dot-Z_1-integral}
  \frac{i}{2\pi^3} \int_{Z_1 \in U(2)_{R_1}}
  \frac{N(Z_1-Y_1) \cdot N(Z_1-Y_2) \cdot \ldots \cdot N(Z_1-Y_{k-1})}
       {N(Z_1-Y_k) \cdot N(Z_1-Y_{k+1}) \cdot \ldots \cdot N(Z_1-Y_{2k-1})}
       \cdot N(Z_1)^{-2} \, dV_1,
  \end{equation}
  where $Y_1, \dots, Y_{2k-1}$ are some labels of the diagram
  ($T_j$'s, $Z_2$, $W_1$ or $W_2$), possibly with repetitions.

Recall the matrix coefficient expansion (\ref{1/N-expansion}):
$$
\frac 1{N(Z_1-Y_j)}= N(Z_1)^{-1} \cdot \sum_{l,m,n}
t^l_{n\,\underline{m}}(Y_j) \cdot t^l_{m\,\underline{n}}(Z_1^{-1}),
\qquad \begin{matrix} j=k,\dots,2k-1, \quad l=0,\frac 12, 1, \frac 32,\dots, \\
m,n = -l, -l+1, \dots, l. \end{matrix}
$$
Then
$$
\frac1{N(Z_1-Y_k) \cdot N(Z_1-Y_{k+1}) \cdot \ldots \cdot N(Z_1-Y_{2k-1})} =
\frac1{N(Z_1)^k} + \text{ lower degree terms in $Z_1$}.
$$
On the other hand,
$$
N(Z_1-Y_1) \cdot N(Z_1-Y_2) \cdot \ldots \cdot N(Z_1-Y_{k-1})
= N(Z_1)^{k-1} + \text{ lower degree terms in $Z_1$}.
$$
Hence
$$
\frac{N(Z_1-Y_1) \cdot N(Z_1-Y_2) \cdot \ldots \cdot N(Z_1-Y_{k-1})}
{N(Z_1-Y_k) \cdot N(Z_1-Y_{k+1}) \cdot \ldots \cdot N(Z_1-Y_{2k-1})}
= \frac1{N(Z_1)} + \text{ lower degree terms in $Z_1$}.
$$
Comparing this with the orthogonality relations (\ref{orthogonality}),
we see that the integral (\ref{L-dot-Z_1-integral}) is zero.
Thus $\bar L^{(n)}$ annihilates the first family of generators in
(\ref{2-families-gen}) and hence $I_2^- \otimes \Zh$.
The case of $\Zh \otimes I_2^-$ is similar.
\end{proof}

\begin{prop}  \label{polynomial-prop}
For each $f_1, f_2 \in \Zh$, the function
$\bar L^{(n)} (f_1 \otimes f_2)(W_1,W_2)$ is polynomial, hence an element of
$\Zh^+ \otimes \Zh^+$. Moreover, the operator
$$
\bar L^{(n)}: (\varpi_2^l, \Zh) \otimes (\varpi_2^r, \Zh)
\to (\pi^0_l, \Zh^+) \otimes (\pi^0_r, \Zh^+)
$$
is $U(2) \times U(2)$ and $\mathfrak{gl}(2,\HC)$-equivariant.
\end{prop}

\begin{proof}
Note that the equivariance part is established in
Lemma \ref{L-bar_equivariance}.
Thus, we only need to prove that $\bar L^{(n)} (f_1 \otimes f_2)$ lies in
$\Zh^+ \otimes \Zh^+$.
We do it by induction on $n$ -- the number of solid vertices (or loops)
in the diagram. If $n=1$,
$$
l^{(1)}(Z_1,Z_2;W_1,W_2) =
\frac i{2\pi^3} \int_{T \in U(2)_r}
\frac{dV}{N(Z_1-T) \cdot N(Z_2-T) \cdot N(W_1-T) \cdot N(W_2-T)}.
$$
Recall that $\Zh$ is spanned by elements of the form
$N(Z)^k \cdot t^l_{n\,\underline{m}}(Z)$. So, without loss of generality we
may assume that
$$
f(Z_1) = N(Z_1)^{k_1} \cdot t^{l_1}_{n_1\,\underline{m_1}}(Z_1), \qquad
f(Z_2) = N(Z_2)^{k_2} \cdot t^{l_2}_{n_2\,\underline{m_2}}(Z_2).
$$
When computing $\bar L^{(1)} (f_1 \otimes f_2)(W_1,W_2)$,
we can switch the order of integration -- integrate out $Z_1$, $Z_2$ first
and $T$ later. Using the matrix coefficient expansion (\ref{1/N-expansion}):
$$
\frac 1{N(Z_i-T)}= N(Z_i)^{-1} \cdot \sum_{l,m,n}
t^l_{m\,\underline{n}}(Z_i^{-1}) \cdot t^l_{n\,\underline{m}}(T),
\qquad \begin{matrix} i=1,2, \quad l=0,\frac 12, 1, \frac 32,\dots, \\
m,n = -l, -l+1, \dots, l, \end{matrix}
$$
and orthogonality relations (\ref{orthogonality}),
\begin{multline*}
\Bigl(\frac{i}{2\pi^3}\Bigr)^2
\iint_{\genfrac{}{}{0pt}{}{Z_1 \in U(2)_{R_1}}{Z_2 \in U(2)_{R_2}}}
\frac{f_1(Z_1) \cdot f_2(Z_2) \,dV_1 dV_2}
{N(Z_1-T) \cdot N(Z_2-T) \cdot N(W_1-T) \cdot N(W_2-T)}  \\
= \frac{t^{l_1}_{n_1\,\underline{m_1}}(T) \cdot t^{l_2}_{n_2\,\underline{m_2}}(T)}
{N(W_1-T) \cdot N(W_2-T)} \qquad \text{if $k_1=k_2=-1$}
\end{multline*}
and zero otherwise.
When we integrate out the $T$ variable, we use the matrix coefficient
expansion (\ref{1/N-expansion}) again:
$$
\frac 1{N(W_i-T)}= N(T)^{-1} \cdot \sum_{l,m,n}
t^l_{m\,\underline{n}}(T^{-1}) \cdot t^l_{n\,\underline{m}}(W_i),
\qquad \begin{matrix} i=1,2, \quad l=0,\frac 12, 1, \frac 32,\dots, \\
m,n = -l, -l+1, \dots, l, \end{matrix}
$$
and obtain:
\begin{multline*}
\bar L^{(1)} (f_1 \otimes f_2)(W_1,W_2)
= \sum_{\genfrac{}{}{0pt}{}{l,m,n}{l',m',n'}}
t^l_{n\,\underline{m}}(W_1) \cdot t^{l'}_{n'\,\underline{m'}}(W_2) \\
\times \frac i{2\pi^3} \int_{T \in U(2)_r}
t^{l_1}_{n_1\,\underline{m_1}}(T) \cdot t^{l_2}_{n_2\,\underline{m_2}}(T)
\cdot N(T)^{-2} \cdot t^l_{m\,\underline{n}}(T^{-1}) \cdot
t^{l'}_{m'\,\underline{n'}}(T^{-1}) \,dV.
\end{multline*}
By the orthogonality conditions (\ref{orthogonality}), these integrals
can be non-zero only when the degree of
$t^l_{m\,\underline{n}}(T^{-1}) \cdot t^{l'}_{m'\,\underline{n'}}(T^{-1})$
is negative that of
$t^{l_1}_{n_1\,\underline{m_1}}(T) \cdot t^{l_2}_{n_2\,\underline{m_2}}(T)$,
in other words, $l+l'=l_1+l_2$.
Thus, there are only finitely many non-zero terms and
$$
\bar L^{(1)} (f_1 \otimes f_2)(W_1,W_2) \in \Zh^+ \otimes \Zh^+.
$$

Now we prove the inductive step.
For concreteness, let us assume that the last slingshot is attached to an
$(n-1)$-loop box diagram $d^{(n-1)}$ so that $Z_1$ becomes a solid vertex and
gets relabeled as $T_n$ (the other cases are similar). Then
$$
l^{(n)}(Z_1,Z_2;W_1,W_2) = \frac{i}{2\pi^3} \int_{T_n \in U(2)_{r_n}}
\frac{N(Z_2-W_2) \cdot l^{(n-1)}(T_n,Z_2;W_1,W_2)}
{N(Z_1-T_n) \cdot N(Z_2-T_n) \cdot N(W_2-T_n)} \,dV_{T_n},
$$
where $l^{(n-1)}(Z_1,Z_2;W_1,W_2)$ is the conformal four-point integral
corresponding to the $(n-1)$-loop diagram $d^{(n-1)}$.
By induction, we assume that the result holds for $\bar L^{n-1}$:
$$
\bar L^{(n-1)}: \Zh \otimes \Zh \to \Zh^+ \otimes \Zh^+.
$$
We want to integrate out $Z_1$ first and reduce the integral to $\bar L^{(n-1)}$.
By the matrix coefficient expansion (\ref{1/N-expansion}) and
orthogonality conditions (\ref{orthogonality}),
$$
g_1(T_n) =
\frac{i}{2\pi^3} \int_{Z_1 \in U(2)_{R_1}} \frac{f_1(Z_1)\,dV_1}{N(Z_1-T_n)}
$$
is a (harmonic) polynomial in $T_n$, hence an element of $\Zh^+$.
Using the expansion (\ref{1/N-expansion}) again,
$$
\frac1{N(Z_2-T_n) \cdot N(W_2-T_n)} =
\frac1{N(Z_2) \cdot N(T_n)} \sum_{\genfrac{}{}{0pt}{}{l,m,n}{l',m',n'}}
t^l_{n\,\underline{m}}(T_n) \cdot t^{l'}_{n'\,\underline{m'}}(W_2)
\cdot t^l_{m\,\underline{n}}(Z_2^{-1}) \cdot t^{l'}_{m'\,\underline{n'}}(T_n^{-1}).
$$
Relabeling $T_n$ as $Z_1$, we obtain:
\begin{multline*}
\bar L^{(n)} (f_1 \otimes f_2)(W_1,W_2)
= \sum_{\genfrac{}{}{0pt}{}{l,m,n}{l',m',n'}} t^{l'}_{n'\,\underline{m'}}(W_2)  \\
\times
\bar L^{(n-1)} \left( g_1(Z_1) \cdot f_2(Z_2) \cdot
\frac{N(Z_2-W_2)}{N(Z_1) \cdot N(Z_2)} \cdot t^l_{n\,\underline{m}}(Z_1) \cdot
t^{l'}_{m'\,\underline{n'}}(Z_1^{-1}) \cdot t^l_{m\,\underline{n}}(Z_2^{-1}) \right).
\end{multline*}
By Lemma \ref{I^--annihilated-lem}, $\bar L^{(n-1)}$ annihilates
$\Zh \otimes I_2^-$ and, in particular, annihilates all those elements
\begin{equation}  \label{elements}
g_1(Z_1) \cdot f_2(Z_2) \cdot \frac{N(Z_2-W_2)}{N(Z_1) \cdot N(Z_2)} \cdot
t^l_{n\,\underline{m}}(Z_1) \cdot t^{l'}_{m'\,\underline{n'}}(Z_1^{-1}) \cdot
t^l_{m\,\underline{n}}(Z_2^{-1})
\end{equation}
that have $Z_2$-degree strictly less than $-2$.
Thus, only finitely many indices $l'$, $m'$ and $n'$ contribute
non-zero terms to the above sum.
By Lemma \ref{I^--annihilated-lem}, $\bar L^{(n-1)}$ also annihilates
$I_2^- \otimes \Zh$ and, in particular, annihilates all those elements
(\ref{elements}) that have $Z_1$-degree strictly less than $-2$.
Thus, the above sum has only finitely many non-zero terms, and,
by the induction hypothesis, each term lies in $\Zh^+ \otimes \Zh^+$.
We conclude that
$$
\bar L^{(n)} (f_1 \otimes f_2)(W_1,W_2) \in \Zh^+ \otimes \Zh^+.
$$
\end{proof}

\begin{rem}
It is also true that $\bar L^{(n)}$ increases the total degree of
$f_1 \otimes f_2$ by four. This can be shown by slightly adjusting
the proof of Proposition \ref{polynomial-prop}
\end{rem}

\subsection{Proof of Theorem \ref{magic}}

We can switch the roles of $Z$'s and $W$'s and obtain an operator similar
to $L^{(n)}$ that integrates out the $W_1$, $W_2$ variables:
\begin{multline*}
\Bigl(\frac{i}{2\pi^3}\Bigr)^{-2} \cdot
\acute{L}^{(n)} (\phi_1 \otimes \phi_2)(Z_1,Z_2) \\
= 
\iint_{\genfrac{}{}{0pt}{}{W_1 \in U(2)_{R'_1}}{W_2 \in U(2)_{R'_2}}}
l^{(n)}(Z_1,Z_2;W_1,W_2) \cdot (\degt_{W_1} \phi_1)(W_1)
\cdot (\degt_{W_2} \phi_2)(W_2) \,\frac{dV_1}{N(W_1)} \frac{dV_2}{N(W_2)},
\end{multline*}
where $\phi_1, \phi_2 \in {\cal H}^-$, $R'_1 < r_{\text{min},1}$,
$R'_2 < r_{\text{min},2}$, $Z_1 \in \BB D^-_{r_{\text{max},1}}$,
$Z_2 \in \BB D^-_{r_{\text{max},2}}$
(recall that $r_{\text{max},i}$ and $r_{\text{min},i}$ are defined in
(\ref{r_max}) and (\ref{r_min})).

\begin{prop}  \label{L'-prop}
For each $\phi_1, \phi_2 \in {\cal H}^-$, the function
$\acute{L}^{(n)} (\phi_1 \otimes \phi_2)(Z_1,Z_2)$ belongs to
${\cal H}^- \otimes {\cal H}^-$.
Moreover, the operator
$$
\acute{L}^{(n)}: (\pi^0_r, {\cal H}^-) \otimes (\pi^0_l, {\cal H}^-)
\to (\pi^0_r, {\cal H}^-) \otimes (\pi^0_l, {\cal H}^-)
$$
is $U(2) \times U(2)$ and $\mathfrak{gl}(2,\HC)$-equivariant.
\end{prop}

\begin{proof}
First, we show that
$\acute{L}^{(n)} (\phi_1 \otimes \phi_2) \in {\cal H}^- \otimes {\cal H}^-$.
Suppose that the conformal four-point integral $l^{(n)}(Z_1,Z_2;W_1,W_2)$
corresponds to a box diagram $d^{(n)}$.
We apply
$\bigl(\begin{smallmatrix} 0 & 1 \\ 1 & 0 \end{smallmatrix}\bigr) \in GL(2,\HC)$
in order to reduce the result to Proposition \ref{Prop12}
(Proposition 12 in \cite{L2}). Then
$$
\pi\bigl(\begin{smallmatrix} 0 & 1 \\ 1 & 0 \end{smallmatrix}\bigr)Z = Z^{-1},
\qquad 
\pi^0_l\bigl(\begin{smallmatrix} 0 & 1 \\ 1 & 0 \end{smallmatrix}\bigr) \phi(Z)
= \pi^0_r\bigl(\begin{smallmatrix} 0 & 1 \\ 1 & 0 \end{smallmatrix}\bigr)
\phi(Z) = N(Z)^{-1} \cdot \phi(Z^{-1})
$$
and, by Lemma 10 in \cite{FL1},
$$
N(Z^{-1}-W^{-1}) = N(Z)^{-1} \cdot N(Z-W) \cdot N(W)^{-1}.
$$
Note that the map $Z \mapsto Z^{-1}$ reverses the partial ordering
of vertices in the diagram $d^{(n)}$.
Combining observations (\ref{1edge})-(\ref{4edges}) with Lemma 61 in \cite{FL1},
we obtain
\begin{equation}  \label{l-inversion}
l^{(n)}(Z_1^{-1},Z_2^{-1};W_1^{-1},W_2^{-1})
= N(W_1) \cdot N(W_2) \cdot \tilde l^{(n)}(W_1,W_2;Z_1,Z_2)
\cdot N(Z_1) \cdot N(Z_2),
\end{equation}
where $\tilde l^{(n)}(Z_1,Z_2;W_1,W_2)$ is the conformal four-point integral
corresponding to the box diagram $\tilde d^{(n)}$ obtained from
$d^{(n)}$ by reversing the partial ordering.
Let
$$
\tilde\phi_i(Z_i) = N(Z_i)^{-1} \cdot \phi_i(Z_i^{-1}), \qquad i=1,2;
$$

Then, using (\ref{inversion-symmetry}),
\begin{multline*}
\Bigl(\frac{i}{2\pi^3}\Bigr)^{-2} \cdot (\pi^0_r \otimes \pi^0_l)
\bigl(\begin{smallmatrix} 0 & 1 \\ 1 & 0 \end{smallmatrix}\bigr)
\bigl( \acute{L}^{(n)} (\phi_1 \otimes \phi_2) \bigr) (Z_1,Z_2)  \\
= \iint_{\genfrac{}{}{0pt}{}{W_1 \in U(2)_{R'_1}}{W_2 \in U(2)_{R'_2}}}
\frac{l^{(n)}(Z_1^{-1},Z_2^{-1};W_1^{-1},W_2^{-1}) \cdot
(\degt_{W_1} \tilde\phi_1)(W_1) \cdot (\degt_{W_2} \tilde\phi_2)(W_2)}
{N(Z_1) \cdot N(Z_2) \cdot N(W_1) \cdot N(W_2)}
\,\frac{dV_1}{N(W_1)} \frac{dV_2}{N(W_2)}  \\
= \iint_{\genfrac{}{}{0pt}{}{W_1 \in U(2)_{R'_1}}{W_2 \in U(2)_{R'_2}}}
\tilde l^{(n)}(W_1,W_2;Z_1,Z_2) \cdot
(\degt_{W_1} \tilde\phi_1)(W_1) \cdot (\degt_{W_2} \tilde\phi_2)(W_2)
\,\frac{dV_1}{N(W_1)} \frac{dV_2}{N(W_2)}  \\
= \Bigl(\frac{i}{2\pi^3}\Bigr)^{-2} \cdot
L^{(n)} (\tilde\phi_1 \otimes \tilde\phi_2)(Z_1,Z_2).
\end{multline*}
By Proposition \ref{Prop12}, the last expression lies in
${\cal H}^+ \otimes {\cal H}^+$.
Therefore,
$\acute{L}^{(n)} (\phi_1 \otimes \phi_2) \in {\cal H}^- \otimes {\cal H}^-$.

The proof of the equivariance part is similar to that of
Lemma \ref{L-bar_equivariance}; it is based on Lemma \ref{conformal}
and invariance of the pairing (\ref{H-pairing2}).
\end{proof}

We have the following relation between operators $\bar L^{(n)}$ and
$\acute{L}^{(n)}$:
\begin{multline}  \label{L-dot-prime-relation}
\iint_{\genfrac{}{}{0pt}{}{W_1 \in U(2)_{R'_1}}{W_2 \in U(2)_{R'_2}}}
\bar L^{(n)} (f_1 \otimes f_2)(W_1,W_2)
\cdot (\degt_{W_1} \phi_1)(W_1) \cdot (\degt_{W_2} \phi_2)(W_2)
\, \frac{dV_{W_1}}{N(W_1)} \frac{dV_{W_2}}{N(W_2)}  \\
= \Bigl(\frac{i}{2\pi^3}\Bigr)^2
\iint_{\genfrac{}{}{0pt}{}{Z_1 \in U(2)_{R_1} ,\: Z_2 \in U(2)_{R_2}}
{W_1 \in U(2)_{R'_1} ,\: W_2 \in U(2)_{R'_2}}}
l^{(n)}(Z_1,Z_2;W_1,W_2) \cdot f_1(Z_1) \cdot f_2(Z_2) \\
\times (\degt_{W_1} \phi_1)(W_1) \cdot (\degt_{W_2} \phi_2)(W_2)
\, dV_{Z_1} dV_{Z_2} \frac{dV_{W_1}}{N(W_1)} \frac{dV_{W_2}}{N(W_2)}  \\
= \iint_{\genfrac{}{}{0pt}{}{Z_1 \in U(2)_{R_1}}{Z_2 \in U(2)_{R_2}}}
\acute{L}^{(n)} (\phi_1 \otimes \phi_2)(Z_1,Z_2) \cdot f_1(Z_1) \cdot f_2(Z_2)
\, dV_{Z_1} dV_{Z_2},
\end{multline}
for all $f_1, f_2 \in \Zh$, $\phi_1, \phi_2 \in {\cal H}^-$.
The radii should be chosen so that $R_1 > r_{\text{max},1}$,
$R_2 > r_{\text{max},2}$, $R'_1 < r_{\text{min},1}$ and $R'_2 < r_{\text{min},2}$,
as before.

\begin{cor} \label{harm-harm-cor}
  Suppose that $F(Z_1,Z_2) \in \Zh \otimes \Zh$ is such that
  $\bar L^{(n)}(F)$ is not zero and lies in
  ${\cal H}^+ \otimes {\cal H}^+$, then the projection of $F$ onto
  $\bigl( \Zh/(I_2^- \oplus \Zh^+) \bigr) \otimes
  \bigl( \Zh/(I_2^- \oplus \Zh^+) \bigr)$ is not zero.
\end{cor}

\begin{proof}
Suppose that $\bar L^{(n)}(F)$ is not zero and lies in
${\cal H}^+ \otimes {\cal H}^+$.
Since the pairing (\ref{H-pairing2}) is non-degenerate, there exist
$\phi_1, \phi_2 \in {\cal H}^-$ such that
$$
\iint_{\genfrac{}{}{0pt}{}{W_1 \in U(2)_{R'_1}}{W_2 \in U(2)_{R'_2}}}
\bar L^{(n)}(F)(W_1,W_2)
\cdot (\degt_{W_1} \phi_1)(W_1) \cdot (\degt_{W_2} \phi_2)(W_2)
\, \frac{dV_{W_1}}{N(W_1)} \frac{dV_{W_2}}{N(W_2)} \ne 0.
$$
By (\ref{L-dot-prime-relation}),
\begin{equation}  \label{int-not-zero}
\iint_{\genfrac{}{}{0pt}{}{Z_1 \in U(2)_{R_1}}{Z_2 \in U(2)_{R_2}}}
\acute{L}^{(n)} (\phi_1 \otimes \phi_2)(Z_1,Z_2) \cdot F(Z_1,Z_2)
\, dV_{Z_1} dV_{Z_2} \ne 0.
\end{equation}
On the other hand, by Proposition \ref{L'-prop}, 
$\acute{L}^{(n)} (\phi_1 \otimes \phi_2)(Z_1,Z_2)
\in {\cal H}^- \otimes {\cal H}^-$,
and, by the orthogonality relations (\ref{orthogonality}),
the expression (\ref{int-not-zero}) is zero
unless the projection of $F$ onto
$\bigl( \Zh/(I_2^- \oplus \Zh^+) \bigr) \otimes
\bigl( \Zh/(I_2^- \oplus \Zh^+) \bigr)$ does not vanish.
\end{proof}

\begin{lem}  \label{Zh^+-annihilated-lem}
Suppose that $l^{(n)}(Z_1,Z_2;W_1,W_2)$ is harmonic with respect to $W_1$ or
$W_2$, then the operator $\bar L^{(n)}$ annihilates $\Zh^+ \otimes \Zh$ and
$\Zh \otimes \Zh^+$.
\end{lem}

\begin{proof}
  We will prove that $\bar L^{(n)}$ annihilates $\Zh^+ \otimes \Zh$,
  the other case is similar.
  Without loss of generality we can assume that $F \in \Zh^+ \otimes \Zh$
  is homogeneous.
  (We can even assume that $F$ is of the form $f_1 \otimes f_2$ for some
  $f_1 \in \Zh^+$, $f_2 \in \Zh$.)
  The proof is by induction on the total degree of $F \in \Zh^+ \otimes \Zh$.
  By Lemma \ref{I^--annihilated-lem}, $\bar L^{(n)}$ annihilates
  $\Zh^+ \otimes I_2^-$.
  Thus, if $\bar L^{(n)}(F) \ne 0$, the total degree of $F$ is at least $-2$.
  Since $\bar L^{(n)}$ is $\mathfrak{gl}(2,\HC)$-equivariant,
  \begin{equation}  \label{equiv-eqn}
  \bar L^{(n)} \bigl( (\partial_{ij})_{Z_1} F \bigr)
  + \bar L^{(n)} \bigl( (\partial_{ij})_{Z_2} F \bigr) =
  (\partial_{ij})_{W_1} \bar L^{(n)}(F) + (\partial_{ij})_{W_2} \bar L^{(n)}(F),
  \qquad 1 \le i,j \le 2.
  \end{equation}
  Now, suppose that $F$ is homogeneous, $\bar L^{(n)}(F) \ne 0$ and
  that $F$ has the lowest total degree among such functions. Then
  $(\partial_{ij})_{Z_1} F$ as well as $(\partial_{ij})_{Z_2} F$
  have strictly lower degrees, hence 
  $$
  \bar L^{(n)} \bigl( (\partial_{ij})_{Z_1} F \bigr) =
  \bar L^{(n)} \bigl( (\partial_{ij})_{Z_2} F \bigr) = 0,
  \qquad 1 \le i,j \le 2,
  $$
  and the right hand side of (\ref{equiv-eqn}) is zero as well.
  Therefore, $\bar L^{(n)}(F)$ is a function of $W_1-W_2$.
  Since $\bar L^{(n)}(F)(W_1,W_2)$ is harmonic with respect to $W_1$ or $W_2$,
  it is also harmonic with respect to the other variable, i.e.
  $\bar L^{(n)}(F) \in {\cal H}^+ \otimes {\cal H}^+$,
  and we get a contradiction with Corollary \ref{harm-harm-cor}.
\end{proof}

\begin{cor}  \label{harmonic-input+output-cor}
Suppose that $l^{(n)}(Z_1,Z_2;W_1,W_2)$ is harmonic with respect to $W_1$ or
$W_2$, then the operator
$\bar L^{(n)}: \Zh \otimes \Zh \to {\cal H}^+ \otimes \Zh^+$
descends to the quotient
$$
\bigl( \Zh/(I_2^- \oplus \Zh^+) \bigr) \otimes
\bigl( \Zh/(I_2^- \oplus \Zh^+) \bigr),
$$
and its image lies in ${\cal H}^+ \otimes {\cal H}^+$.
\end{cor}

\begin{proof}
The part that $\bar L^{(n)}$ descends to
$\bigl( \Zh/(I_2^- \oplus \Zh^+) \bigr) \otimes
\bigl( \Zh/(I_2^- \oplus \Zh^+) \bigr)$
follows immediately from  Lemmas \ref{I^--annihilated-lem} and
\ref{Zh^+-annihilated-lem}.

It remains to show that the image of $\bar L^{(n)}$ lies in
${\cal H}^+ \otimes {\cal H}^+$. By Proposition \ref{quotient-prop},
$$
\bigl(\varpi_2^*, \Zh/(I_2^- \oplus \Zh^+)\bigr) \simeq (\pi^0_*, {\cal H}^+),
$$
where $*$ stands for $l$ or $r$, with the isomorphism map (\ref{harm-iso}).
By (\ref{L-L-bar-rel}), the images of $\bar L^{(n)}$ and $L^{(n)}$ are the same.
Then the result follows from Proposition \ref{Prop12}.
\end{proof}

\begin{cor}  \label{harmonic-cor}
Suppose that $l^{(n)}(Z_1,Z_2;W_1,W_2)$ is harmonic with respect to $W_1$ or
$W_2$, then $l^{(n)}(Z_1,Z_2;W_1,W_2)$ is harmonic in all the variables
$$
Z_1 \in \BB D^-_{r_{\text{max},1}}, \quad Z_2 \in \BB D^-_{r_{\text{max},2}}, \quad
W_1 \in \BB D^+_{r_{\text{min},1}}, \quad W_2 \in \BB D^+_{r_{\text{min},2}}.
$$
\end{cor}

\begin{proof}
Recall that the restrictions of the functions
$N(Z)^k \cdot t^l_{n\,\underline{m}}(Z)$'s to $U(2)$ form a complete orthogonal
basis of functions on $U(2)$.
Also, a complex analytic function defined on a connected open subset
$\Omega$ of $\HC$ is uniquely determined by its restriction to
$\Omega \cap U(2)_R$ whenever this intersection is non-empty and $R>0$.
If $l^{(n)}(Z_1,Z_2;W_1,W_2)$ is not harmonic with respect to $Z_1$ or $Z_2$,
by the orthogonality relations (\ref{orthogonality}) there will be a pair of
functions $f_1, f_2 \in \Zh$ with $f_1$ or $f_2$ lying in $I_2^- \oplus \Zh^+$
such that $\bar L^{(n)}(f_1 \otimes f_2) \ne 0$, which contradicts
Corollary \ref{harmonic-input+output-cor}.
If $l^{(n)}(Z_1,Z_2;W_1,W_2)$ is not harmonic with respect to $W_1$ or $W_2$,
there will be $f_1, f_2 \in \Zh$ such that
$\bar L^{(n)}(f_1 \otimes f_2) \notin {\cal H}^+ \otimes {\cal H}^+$, which,
again, contradicts Corollary \ref{harmonic-input+output-cor}.
\end{proof}

Now, we drop the assumption that $l^{(n)}(Z_1,Z_2;W_1,W_2)$ is harmonic
with respect to $W_1$ or $W_2$.

\begin{prop}   \label{harmonic-input+output}
The operator $\bar L^{(n)}$ annihilates $\Zh^+ \otimes \Zh$ and
$\Zh \otimes \Zh^+$, it descends to the quotient
$$
\bigl( \Zh/(I_2^- \oplus \Zh^+) \bigr) \otimes
\bigl( \Zh/(I_2^- \oplus \Zh^+) \bigr),
$$
and its image lies in ${\cal H}^+ \otimes {\cal H}^+$.
Additionally, the $n$-loop box integral $l^{(n)}(Z_1,Z_2;W_1,W_2)$ is harmonic
in the variables
$$
Z_1 \in \BB D^-_{r_{\text{max},1}}, \quad Z_2 \in \BB D^-_{r_{\text{max},2}}, \quad
W_1 \in \BB D^+_{r_{\text{min},1}}, \quad W_2 \in \BB D^+_{r_{\text{min},2}}.
$$
\end{prop}

\begin{proof}
If $n=1$, the one-loop box integral $l^{(1)}(Z_1,Z_2;W_1,W_2)$ is given by
(\ref{l^{(1)}}) and harmonic with respect to each variable, thus the
result follows from Lemma \ref{Zh^+-annihilated-lem} and its
Corollaries \ref{harmonic-input+output-cor}, \ref{harmonic-cor}.

Assume that $n \ge 2$, then the $n$-loop box diagram $d^{(n)}$ is obtained by
attaching the last slingshot to an $(n-1)$-loop box diagram $d^{(n-1)}$.
Let us consider the possibilities. If the last slingshot is attached so that 
the vertex $W_1$ in $d^{(n-1)}$ becomes a solid vertex and gets relabeled as
$T_n$, then $l^{(n)}(Z_1,Z_2;W_1,W_2)$ is automatically harmonic with respect
to $W_1$, and the result is true by Lemma \ref{Zh^+-annihilated-lem} and its
Corollaries \ref{harmonic-input+output-cor}, \ref{harmonic-cor}.
Similarly, if the slingshot is attached so that the vertex $W_2$ in $d^{(n-1)}$
becomes a solid vertex, then $l^{(n)}(Z_1,Z_2;W_1,W_2)$ is harmonic with respect
to $W_2$, and we are done.

Finally, suppose that the last slingshot is attached so that either $Z_1$
or $Z_2$ in $d^{(n-1)}$ becomes a solid vertex, then
$l^{(n)}(Z_1,Z_2;W_1,W_2)$ is harmonic with respect to either $Z_1$ or $Z_2$.
We apply
$\bigl(\begin{smallmatrix} 0 & 1 \\ 1 & 0 \end{smallmatrix}\bigr) \in GL(2,\HC)$
and proceed as in the proof of Proposition \ref{L'-prop}.
Recall that the map $Z \mapsto Z^{-1}$ reverses the partial ordering
of vertices in the diagram $d^{(n)}$.
Let $\tilde l^{(n)}(Z_1,Z_2;W_1,W_2)$ be the conformal four-point integral
corresponding to the box diagram $\tilde d^{(n)}$ obtained from
$d^{(n)}$ by reversing the partial ordering.
Then $\tilde l^{(n)}(Z_1,Z_2;W_1,W_2)$ is harmonic with respect to either
$W_1$ or $W_2$ and Corollary \ref{harmonic-cor} is applicable.
Thus, $\tilde l^{(n)}(Z_1,Z_2;W_1,W_2)$ is harmonic with respect to each
variable. By the relation (\ref{l-inversion}), $l^{(n)}(Z_1,Z_2;W_1,W_2)$
is harmonic with respect to $W_1$ and $W_2$, and the result follows.
\end{proof}

Now we can give a proof of Theorem \ref{magic}
(magic identities in the Minkowski metric case).

\begin{proof}
Recall that the restrictions of the functions
$N(Z)^k \cdot t^l_{n\,\underline{m}}(Z)$'s to $U(2)$ form a complete orthogonal
basis of functions on $U(2)$.
Also, a complex analytic function defined on a connected open subset
$\Omega$ of $\HC$ is uniquely determined by its restriction to
$\Omega \cap U(2)_R$ whenever this intersection is non-empty and $R>0$.
Let $l^{(n)}(Z_1,Z_2;W_1,W_2)$ and $\tilde l^{(n)}(Z_1,Z_2;W_1,W_2)$ be two
conformal four-point integrals corresponding to any two $n$-loop box diagrams.
If $l^{(n)}(Z_1,Z_2;W_1,W_2) \ne \tilde l^{(n)}(Z_1,Z_2;W_1,W_2)$, there exist
$f_1 = N(Z)^k \cdot t^l_{n\,\underline{m}}(Z)$ and
$f_2 = N(Z)^{k'} \cdot t^{l'}_{n'\,\underline{m'}}(Z)$ such that
\begin{multline}  \label{not_equal}
\Bigl(\frac{i}{2\pi^3}\Bigr)^2
\iint_{\genfrac{}{}{0pt}{}{Z_1 \in U(2)_{R_1}}{Z_2 \in U(2)_{R_2}}}
l^{(n)}(Z_1,Z_2;W_1,W_2) \cdot f_1(Z_1) \cdot f_2(Z_2) \,dV_1 dV_2  \\
\ne
\Bigl(\frac{i}{2\pi^3}\Bigr)^2
\iint_{\genfrac{}{}{0pt}{}{Z_1 \in U(2)_{R_1}}{Z_2 \in U(2)_{R_2}}}
\tilde l^{(n)}(Z_1,Z_2;W_1,W_2) \cdot f_1(Z_1) \cdot f_2(Z_2) \,dV_1 dV_2.
\end{multline}
In other words, the operators $\bar L^{(n)}$ associated to $l^{(n)}$ and
$\tilde l^{(n)}$ take different values on $f_1 \otimes f_2$.
By Proposition \ref{harmonic-input+output}, $l^{(n)}$ and $\tilde l^{(n)}$
are harmonic in $Z_1 \in \BB D^-_{r_{\text{max},1}}$ and
$Z_2 \in \BB D^-_{r_{\text{max},2}}$, hence $k=k'=-1$
(otherwise both sides of (\ref{not_equal}) are zero).
Then by Theorem \ref{operator-magic} and relation (\ref{L-L-bar-rel}),
both sides of (\ref{not_equal}) are equal, which is a contradiction.
\end{proof}

As an immediate consequence of Theorem \ref{magic}, we also have
a stronger operator version of magic identities
(compare with Theorem \ref{operator-magic}):

\begin{cor}
The integral operators
$\bar L^{(n)} : (\varpi_2^l, \Zh) \otimes (\varpi_2^r, \Zh) \to
(\pi^0_l, {\cal H}^+) \otimes (\pi^0_r, {\cal H}^+)$ associated to any
two $n$-loop box diagrams are the same.
The operators $\bar L^{(n)}$ descend to the quotient
$$
\bigl( \varpi_2^l, \Zh/(I_2^- \oplus \Zh^+) \bigr) \otimes
\bigl( \varpi_2^r, \Zh/(I_2^- \oplus \Zh^+) \bigr)
\simeq
(\pi^0_l, {\cal H}^+) \otimes (\pi^0_r, {\cal H}^+),
$$
and the action of $\bar L^{(n)}$ on each irreducible component of
$(\pi^0_l, {\cal H}^+) \otimes (\pi^0_r, {\cal H}^+)$ is described in
Theorem 14 in \cite{L2}.
\end{cor}

\separate

\noindent
{\em Department of Mathematics, Indiana University,
Rawles Hall, 831 East 3rd St, Bloomington, IN 47405}

\end{document}